\documentclass[11pt,reqno]{amsart}
\usepackage{amsmath}
\usepackage{amsthm}
\usepackage{amssymb}

\usepackage{amsmath}
\usepackage{amsthm}
\usepackage{amssymb}
\usepackage{fancyhdr}
\usepackage{enumerate}
\usepackage{amscd}
\usepackage[all]{xy}
\usepackage{graphicx}

\theoremstyle{remark}
\newtheorem*{Remark}{Remark}

\theoremstyle{plain}
\newtheorem{THM}{Theorem}
\newtheorem{lemma}{Lemma}[section]
\newtheorem{prop}[lemma]{Proposition}
\newtheorem{thm}{Theorem}[section]

\newtheorem{defi}[lemma]{Definition}
\newtheorem*{rmk}{Remark}
\newtheorem{cor}[lemma]{Corollary}
\numberwithin{equation}{section}

\begin{document}
\title{On the Center of Mass of Isolated Systems with General Asymptotics}
\author{Lan-Hsuan Huang}
\date{}
\begin{abstract}
We propose a definition of center of mass for asymptotically flat manifolds satisfying Regge-Teitelboim condition at infinity. This definition has a coordinate-free expression and natural properties. Furthermore, we prove that our definition is consistent both with the one proposed by Corvino and Schoen and another by Huisken and Yau. The main tool is a new density theorem for data satisfying the Regge-Teitelboim condition. 
\end{abstract}
%%% Rick added the last sentence %%%%%
\maketitle
\pagestyle{myheadings}
\markright{On the Center of Mass of Isolated Systems with General Asymptotics}

%%%%%%%%%%%%%%%%%%%%%%%%%%%%%%%%%%%%%%%%%%%%%%%%%%%%%%%%%%%%%%%%%%%%%%%%%%%%%
%%%%%%%%%%%%%%%%%%%%%%%%%%%%%%%%%%%%%%%%%%%%%%%%%%%%%%%%%%%%%%%%%%%%%%%%%%%%%
%%%%%%%%%%%%%%%%%%%%%%%%%%%%%%%%%%%%%%%%%%%%%%%%%%%%%%%%%%%%%%%%%%%%%%%%%%%%%

\section{Introduction}
A $3$-manifold $M$ with a Riemannian metric $g$ and a two-tensor $K$ is called a vacuum initial data set $( M, g ,K )$ if $g$ and $K$ satisfy the constraint equations 
\begin{align} \label{ConE}
		 R_g - | K |_g^2 + H^2 = 0, \notag&&\\
		\mbox{div}_g ( K ) - d H = 0, &&
\end{align}
where $ R_g $ is the scalar curvature of $M$ and $ H = Tr_g ( K ) = g^{ ij } K_{ ij }$. We say $ ( M, g, K) $ is asymptotically flat (AF) if it is a vacuum initial data set and there exists a coordinate $ \{ x \}$ outside a compact set, say $B_{R_0}, $ in $M$ such that, for some $\delta \in (1/2, 1]$, 
\begin{align}\label{AF}
		g_{ij}(x) = \delta_{ij} + h_{ij}(x),  \, h_{ij} = O( |x|^{-\delta}) &\qquad K_{ij}(x) = O( |x|^{-1-\delta} )\notag \\ 
		g_{ij,k}(x) = O ( |x|^{-1-\delta} ) &\qquad K_{ij,k}(x)= O( |x|^{-2-\delta} )\notag \\
		g_{ij,kl}(x) = O ( |x|^{-2-\delta} ) & \qquad K_{ij,kl} (x) = O( |x|^{-3-\delta})
\end{align}
and similarly for higher derivatives, up to the third derivatives on $g$ and up to the second derivatives on $K$. For AF manifolds, the ADM mass $m$ is defined by
\begin{align*}
		m = \frac{1}{ 16\pi } \lim_{r\rightarrow \infty} \int_{ |x| = r} \big( g_{ij,i}- g_{ii,j} \big)\nu_g^j \, d\sigma_g,
\end{align*}
where $ \{ |x| = r \} $ is the Euclidean sphere, $\nu_g$ is the unit outward normal vector field with respect to the metric $g$, and $ d \sigma_g$ is the volume form induced from $(M,g)$. We remark that if one replaces $\nu_g$ by the unit normal vector field  $\nu_0$ with respect to the Euclidean metric and $d \sigma_g$ by the volume form $d\sigma_0$ with respect to the induced metric from the Euclidean space, the limit is in the same. Similarly, the remark holds for \eqref{CTM}, \eqref{ADMmass} and \eqref{CS}.

%%%%%%%%%%%%%%%%%%%%%%%%%%%%%%%%%AF-RT%%%%%%%%%%%%%%%%%%%%%%%%%%%%%%%%%%%%%%%%%%%%%%%%%%%%%%%
For the case in which we are interested, we require asymptotic symmetry on $ ( M, g, K ) $. We say that $( M, g, K )$ is asymptotically flat satisfying Regge-Teitelboim condition (AF-RT) if $ ( M, g, K ) $ is AF  and $g, K$ satisfy these asymptotically even/odd conditions
\begin{align} \label{AF-RT}
g^{odd}_{ij} (x) = O ( |x|^{-1-\delta} ) &\qquad K^{even}_{ij}(x)  = O( |x|^{-2-\delta} )\notag \\
\big( g^{odd}_{ij} \big)_{,k} (x) = O ( |x|^{-2-\delta} ) & \qquad \big( K^{even}_{ij} \big)_{,k} (x) = O( |x|^{-3-\delta} )
\end{align}
and on higher derivatives, where $f^{odd} (x) = f( x ) - f ( -x )$ and $ f^{even} (x) = f( x ) + f( -x )$, \cite{RT}. Notice that $f^{odd}$ and $f^{even}$ are only defined outside $B_{R_0}$ in which the coordinates $\{x\}$ are defined. It is proved by Corvino and Schoen in \cite{CS} that AF-RT manifolds form a dense subset of AF manifolds in some suitable weighted Sobolev spaces. 

%%%%%%%%%%%%%%%%%%%%%%%%%%%%% Our Def of Center of Mass%%%%%%%%%%%%%%%%%%%%%%%%%%%%%%%%%%%%%%%%%%%%%%%%%%%%%%%%%%%
For $( M, g, K )$ satisfying AF-RT, we propose an intrinsic definition of center of mass\footnote{The intrinsic definition can be generalized to $n$-dimensional manifolds if we replace the factor $\frac{1}{ 16\pi m}$ by a suitable constant depending on $n$ and $\omega_n$, where $\omega_n$ is the volume of the unit ball in $\mathbb{R}^n$. Hence the following arguments work more generally for $n$-dimensional manifolds, but for simplicity, we assume $n = 3$.} suggested by Richard Schoen,
\begin{align}
		C_I^{\alpha} = \frac{1}{ 16\pi m}\lim_{ r \rightarrow \infty} \int_{ |x| = r } \left( R_{ij} - \frac{1}{2} R_g g_{ij} \right) Y_{(\alpha)}^i \nu_g^j d{\sigma}_g , \qquad \alpha = 1, 2, 3  \label{CTM}
\end{align}
where $R_{ij}$ is the Ricci curvature of $M$ and $Y_{(\alpha)} =  \big(  | x |^2 \delta^{\alpha i} - 2 x^{\alpha} x^i \big) \frac{\partial}{\partial x^i} $ is a Euclidean conformal Killing vector field. We use the Einstein summation convention and sum over repeated indices; though sometimes we employ summation symbols for clarity. The definition is intrinsic because it is a flux integral at infinity of the three-dimensional Einstein tensor contracted with two vector fields $Y_{(\alpha)}$ and $\nu_g$, and $Y_{(\alpha)}$ has a geometric meaning. Moreover, the surface of integration can be defined more geometrically (proposition \ref{GeneSur}). If $(M, g, K) $ is AF-RT, the limit converges. Note that the above expression is not defined when the ADM mass $m$ is zero. However, the center of mass can not be well-defined when $m$ is zero because the positive mass theorem says that $M$ is actually the Euclidean space, if the scalar curvature $R_g \ge 0$ \cite{SY79, SY81}. 

This intrinsic definition is motivated by a similar expression of the ADM mass when $ Y_{(\alpha)}$ in definition (\ref{CTM}) is replaced by $\{ - 2 x^{i}\frac{\partial  }{ \partial x^i}\}$, the radial direction Euclidean conformal Killing vector field, 
\begin{align} \label{ADMmass}
		m = \frac{1}{ 16\pi } \lim_{r\rightarrow \infty} \int_{ |x| = r}  \left( R_{ij} - \frac{1}{2} R_g g_{ij} \right) ( -2 x^i) \nu_g ^j d{\sigma}_g.
\end{align}
The Euclidean conformal Killing vector fields $\{ - 2 x^{i}\frac{\partial  }{ \partial x^i}\}$ and $ Y_{(\alpha)}$ generate dilation and translation near infinity. A detailed discussion about these two vector fields can be found in \cite{CW}. Another motivation comes from the spatial Schwarzschild metric 
$$
		g^S(x) = \bigg( 1 + \frac{ m }{ 2 | x - p | }\bigg)^4 \delta = 1 + \frac{ 2m }{ | x | } + \frac{ 2m x\cdot p  }{ | x |^3} + \frac{ 3 m }{ 2 | x |^3 } + O( |x|^{-3} ).
$$
If we replace the metric in (\ref{CTM}) by $g^S$, $C_I$ is precisely the vector $p$ which indicates the center of the manifold. It is worth mentioning that although in the Schwarzschild case $p$ is not a point in the manifold, one can develop polar coordinates using concentric spheres centered at $p$. 

Other notions of center of mass have been proposed. Huisken and Yau \cite{HY} defined the center of mass for $(M,g, K)$ which are spherically asymptotically flat (SAF), i.e. $(M, g, K)$ is AF and 
\begin{align} \label{SAF}
		&g_{ij}(x) = \left( 1 + \frac{ 2m }{ | x | }\right) \delta_{ij} + p_{ij}\\ \notag
		& \quad p_{ij}(x) = O( |x|^{-2}), \partial^{\alpha} p_{ij}(x) = O( |x|^{ -2 - |\alpha| }), \quad 1 \le | \alpha| \le 4.  
\end{align}
They proved the existence and uniqueness of the constant mean curvature foliation $\{ M_r\}$ for SAF manifolds, where the mean curvature of $M_r$ is $(2/r)- (4m/r^2) + O(r^{-3})$. They showed that the geometric center defined as follows converges:
\begin{align} \label{HY}
		C^{\alpha}_{HY}=\lim_{r \rightarrow \infty} \frac{ \int_{M_r }z^{\alpha } \, d\sigma_0 }{ \int_{ M_r } \, d\sigma_0 }, \quad \alpha = 1,2,3,
\end{align}
where $ z$ is the position vector of $M_r$ in $M$. This definition is also motivated by the spatial Schwarzschild manifold in which the constant mean curvature foliation is $\{ | x - p | = r\}$.  Another definition by Corvino and Schoen in \cite{CS} is defined for AF-RT manifolds by
\begin{align}
	C_{CS}^{\alpha} &=\frac{1}{ 16 \pi m}  \lim_{ r \rightarrow \infty} \left( \int_{ |x| = r}  x^{\alpha} ( g_{ij,i} - g_{ii,j} ) \nu_g^j d\sigma_g \right.\nonumber\\
	&\qquad \qquad\qquad \qquad \qquad\left.- \int_{ |x| = r} ( h_{i\alpha} \nu_g^i - h_{ii} \nu_g^{\alpha}) \,d\sigma_g \right),\label{CS}
\end{align}
where we recall that $h_{ij} = g_{ij} - \delta_{ij}$ , see also \cite{BO, BLP}. One application of this definition in \cite{CS} is the gluing theorem which allows Corvino and Schoen to approximate AF manifolds by solutions which agree with the original data inside a given region and are identical to a suitable Kerr solution outside a compact set. Their definition is a local-coordinate expression which is natural from the Hamiltonian formulation and convenient for calculation purposes, but may obscure interesting geometry.

The main purpose of this paper is to prove that the intrinsic definition (\ref{CTM}) is equivalent to the Corvino--Schoen definition (\ref{CS}). Moreover, for SAF manifolds in which the unique foliation of constant mean curvature surfaces exists, our intrinsic definition is equal to the Huisken--Yau definition \eqref{HY}. In other words, the intrinsic definition is a coordinate-free expression of the Corvino--Schoen definition and it generalizes the Huisken--Yau definition.
\begin{THM} \label{CorvinoSchoen}
Assume $(M,g, K )$ is AF-RT (\ref{AF-RT}). Then 
$$ 
		C_I = C_{CS}.
$$
\end{THM}
\begin{THM}  \label{HuiskenYau}
Assume $ ( M, g, K  ) $ is SAF (\ref{SAF}). Then 
$$
		C_I = C_{HY}.
$$
\end{THM}
We would like to make a note that Corvino and Wu \cite{CW} recently have a result about the equivalence of these definitions under the assumption that the metric $g$ is conformally flat at infinity with vanishing scalar curvature. In that special case, they are able to derive some explicit estimates. However, it seems that their approach cannot be generalized to AF-RT metrics.

This article is organized as follows. In section 3, we prove a density theorem (theorem \ref{DenThm2}) for AF-RT manifolds. The theorem is crucial for most of the arguments in the paper and may be of independent interest. In section 4, we discuss natural properties of the intrinsic definition. Theorem \ref{CorvinoSchoen} and Theorem \ref{HuiskenYau} are proved in sections 5 and 6, respectively.

%%%%%%%%%%%%%%%%%%%%%%%%%%%%%%%%%%%%%%%%%%%%%%%%%%%%%%%%%%%%%%
%%%%%%%%%%%%%%%%%%%%%%%%%%%%%%%%%%%%%%%%%%%%%%%%%%%%%%%%%%%%%%%%%%%%%%%%%%%%%
%%%%%%%%%%%%%%%%%%%%%%%%%%%%%%%%%%%%%%%%%%%%%%%%%%%%%%%%%%%%%%%%%%%%%%%%%%%%%

%%%%%%%%%%%%%%%%%%%%%%%%%%%%%%%%%%%%%%%%%%%%%%%%%%%%%%%%%%%%%%%%%%%%%%%%%%%%%
%%%%%%%%%%%%%%%%%%%%%%%%%%%%%%%%%%%%%%%%%%%%%%%%%%%%%%%%%%%%%%%%%%%%%%%%%%%%%
%%%%%%%%%%%%%%%%%%%%%%%%%%%%%%%%%%%%%%%%%%%%%%%%%%%%%%%%%%%%%%%%%%%%%%%%%%%%%

\section{The Density Theorem}

Let $ ( M, g, K ) $ be a vacuum initial data set. We introduce the momentum tensor 
$$
		\pi^{ ij } = K^{ ij } - Tr_g ( K ) g^{ij}.
$$
The constraint equations (\ref{ConE}) then take the form 
\begin{align} \label{PiCE}
		R_g + \frac{1}{2} ( Tr_g \pi )^2 - | \pi |_g^2 = 0, \notag&&\\
		\mbox{div}_g ( \pi )  = 0, &&
\end{align}
and we define 
$$ 
		\Phi(g, \pi) = \left(R_g + \frac{1}{2} ( Tr_g \pi )^2 - | \pi |_g^2\,,\, \mbox{div}_g ( \pi ) \right).
$$
In the case  that $ ( M, g, \pi ) $ is AF-RT, Corvino and Schoen \cite{CS} prove that AF manifolds can be approximated by AF-RT manifolds in some weighted Sobolev space. Moreover, instead of requiring smooth solutions to the constraint equations (\ref{PiCE}), their theorem works for solutions with weak regularity. Before we state the theorem, we need the following definitions.
\begin{defi}[Linear and Angular Momentum]
The linear momentum $(P_1, P_2, P_3)$ and the angular momentum $(J_1, J_2, J_3)$ are defined as follows. 
\begin{align*}
		&P_i = \frac{1}{8 \pi } \lim_{r \rightarrow \infty} \int_{|x| = r} \pi_{ij} \nu_g^j \, d \sigma_g,&\\
		&J_{ \alpha} = \frac{1}{8 \pi } \lim_{r \rightarrow \infty} \int_{ |x| = r} \pi_{jk} Z_{(\alpha)}^j \nu_g^k \, d \sigma_g,&
\end{align*}
where $Z_{(\alpha)}$ give the rotation fields in $\mathbb{R}^3$, for example $Z_{(1)} = x^2 \frac{ \partial }{ \partial x^3} - x^3 \frac{ \partial }{ \partial x^2}.$ 
\end{defi}
\begin{rmk}
The linear momentum is well-defined for AF manifolds and the angular momentum is well-defined for AF-RT manifolds.
\end{rmk}
%%%%%%%%%%%%%%%%%%%%Define Weighted Sobolev%%%%%%%%%%%%%%%%%%%%%%%%%%%%%
\begin{defi}[Weighted Sobolev Spaces]
For a non-negative integer $k$, a non-negative real number $ p$, and a real number $\delta$, we say $ f \in W^{k,p}_{-\delta} (M) $ if 
$$
		\| f\|_{W^{k,p}_{-\delta} (M)} \equiv \left( \int_M \sum_{|\alpha| \le k} \left( \big| D^{\alpha} f \big| \rho ^{ | \alpha| + \delta }\right)^p \rho^{-3} \, d \sigma_g  \right)^{\frac{1}{p}} < \infty,
$$
where $\alpha$ is a multi-index and $\rho$ is a continuous function with $\rho = |x|$ on $M\setminus B_{R_0}$. \\
When $p = \infty$, 
$$
		\| f\|_{W^{k,\infty}_{-\delta} (M)} = \sum_{|\alpha | \le k} ess \sup_{M } | D^{\alpha} f | \rho^{ |\alpha| + \delta}.
$$
\end{defi}
%%%%%%%%%%%%%%%%%%%%%% Harmonic asymptotics %%%%%%%%%%%%%%%%%%%%%%%%%%%%%%%%%%%%%%%%%%%%%
\begin{defi}[Harmonic Asymptotics] \label{HA}
$(M, g, \pi)$ is said to have harmonic asymptotics if $(M,g,\pi)$ is AF and
\begin{align}
		g = u^4 \delta,  \quad \pi = u^2 ( \mathcal{L}_{\delta} X ) \label{Approx}
\end{align}
outside a compact set for some $u, X$ tending to $1, 0$ respectively, where for any metric $g$, $\mathcal{L}_g X$ is  the operator associated to the Lie derivative $L_{X} g $ defined by 
$$ 
		\mathcal{L}_g X \equiv L_{X} g - \mbox{div}_{g} ( X)  g .
$$ 
\end{defi}
By the constraint equations (\ref{PiCE}), $u$ and $X$ in definition \ref{HA} satisfy the following equations outside the compact set,
\begin{align*}
		& 8 \Delta_{\delta} u = \Big( - | \mathcal{L}_{\delta} X |^2 + \frac{1}{2} \big( Tr_{\delta} ( \mathcal{L}_{\delta} X )\big)^2 \Big) u,& \\
		& \Delta_{\delta} X^i + 4 u^{-1} u_j ( \mathcal{L}_{ \delta } X )^i_{ j } - 2 u^{-1} u^i Tr_{ \delta }( \mathcal{ L }_{ \delta } X ) = 0,&  
\end{align*} 
where $u_i= u^i = \frac{\partial u}{ \partial x_i}$. Asymptotic flatness requires $u, X$ tend to $1, 0$, respectively, at some decay rate. Using the decay conditions on $u$ and $ X$, we have $ \Delta_{\delta} u  = O( |x|^{ -2 - 2 \delta}) $ and $ \Delta_{\delta} X_i  = O( |x|^{ -2 -2\delta} )$.  As shown in \cite{Bartnik}, the asymptotic behavior implies that outside a compact set, 
\begin{eqnarray}
		 u = 1 + \frac{ a }{ |x | }+ O( |x|^{- 1 - \delta } ), \quad X^i = \frac{ b^i }{ | x | } + O( |x|^{ -1 - \delta}),  \label{Asym1}
\end{eqnarray}
for some constants $a$ and $b^i$. Note that this clearly implies that $( M, g , \pi )$ is AF-RT. Furthermore, if $(M, g,\pi)$ is AF-RT, $u^{odd}$ and $ (X^i)^{odd}$ satisfy equations with better decay: $ \Delta_{ \delta } u^{ odd } = O( | x |^{ - 4 - 2\delta } ) $ and $\Delta_{\delta} ( X^i )^{odd} = O( | x |^{ - 4 - 2\delta } ) $. Hence 
\begin{align}
		&u^{odd} (x) = \frac{ c \cdot x }{ |x|^3 }+ O\big( | x |^{ -2 - \delta } \big),  \quad &&(X^i)^{odd} ( x ) = \frac{ d_{(i)} \cdot x }{ | x |^3 } + O\big( | x |^{ -2 - \delta } \big), \label{Asym2} \notag \\
		&\left( u^{odd} (x) - \frac{ c \cdot x }{ |x|^3 }\right)_{,k} = O(|x|^{-3-\delta}), \quad&&\left( (X^i)^{odd} ( x ) - \frac{ d_{(i)} \cdot x }{ | x |^3 } \right)_{,k} = O(|x|^{-3-\delta}),
\end{align}
for some vectors $c, d_{(i)}$ which are quantities corresponding to the center of mass and angular momentum of $( \bar{g}, \bar{ \pi })$. Since we assume that $g$ and $K$ satisfy the pointwise regularity, the above identities hold pointwisely. However, we can generalize the above discussions to the setting of the weighted Sobolev spaces and have more general results as follows.

%%%%%%%%%%%%%%%%%%%%Corvino--Schoen's Density theorem%%%%%%%%%%%%%%%%%%%%%%%%%%%%%%%%%%%%%%%%%%%%%
\begin{thm}\cite[Theorem 1]{CS}\label{DenThm}
Let $ ( g_{ij} - \delta_{ij} , \pi_{ij} ) \in W^{3,p}_{-\delta} ( M )  \times W^{1,p}_{-1 -\delta} ( M ) $ be a vacuum initial data set, where $\delta \in ( \frac{1}{2} , 1 ) $  and $ p > \frac{3}{2} $. Given any $ \epsilon > 0 $, there exist $k_0 > 0$ and a sequence of solutions $ ( \bar{ g }_k , \bar{ \pi }_k ) $ with harmonic asymptotics satisfying (\ref{Approx}), (\ref{Asym1}) so that
$$
		\| g - \bar{g}_k \|_{W^{3,p}_{-\delta} ( M ) } \le \epsilon , \quad \| \pi - \bar{\pi}_k \|_{W^{1,p}_{-1-\delta} ( M ) } \le \epsilon, \quad \mbox{ for } k \ge k_0.
$$
Moreover, the mass and the linear momentum of $ ( \bar{ g }_k, \bar{ \pi }_k)$ are within $ \epsilon $ of those of $ ( g, \pi )$. 
\end{thm}
The theorem says that the solutions with harmonic asymptotics (\ref{Approx}) are dense among general solutions. More remarkably, the mass and the linear momentum which can be explicitly expressed for solutions with harmonic asymptotics converge to the original initial data set in these weighted Sobolev spaces. However, in the above theorem the center of mass does not seem to converge, neither is the center of mass $C_I$ defined generally for AF manifolds.  Therefore, we would like to modify their theorem and prove, in some weighted Sobolev space, solutions with harmonic asymptotics form a dense subset inside AF-RT solutions so that the centers of mass and the angular momentum converge. The precise statement is as follows:

%%%%%%%%%%%%%%%%%%%%%%%%%% Our Density Theorem %%%%%%%%%%%%%%%%%%%%%%%%%%%%%%%%%%%%%%%%%%%%%%%%%%%
\begin{thm}[Density Theorem]{\label{DenThm2}}
Let $ ( g - \delta , \pi ) \in W^{ 3,p }_{ - \delta} (M) \times W^{ 1,p }_{ -1- \delta} (M) $ be a vacuum initial data set and  $ ( g^{odd} , \pi^{ even }) \in W^{ 3,p }_{ -1 - \delta} ( M \setminus B_{ R_0 })  \times W^{1,p }_{ -2- \delta} ( M \setminus B_{ R_0 }) $, where $\delta \in ( \frac{1}{2} , 1 ) $  and $ p>3 $. Given any $ \epsilon > 0 $ and $\delta_0 \in ( 0, \delta)$, there exist $R $, $k_0 = k_0(R)$, and a sequence of solutions $ ( \bar{ g }_k , \bar{ \pi }_k ) $ with harmonic asymptotics satisfying (\ref{Approx}), (\ref{Asym1}), (\ref{Asym2}) so that $ ( \bar{ g }_k , \bar{ \pi }_k ) $  is within an $\epsilon$-neighborhood of $ ( g, \pi ) $ in the $W^{ 3,p }_{ - \delta} (M) \times W^{ 1,p }_{ -1 - \delta} (M) $ norm and
$$
		 \| \bar{g}_k^{odd} \|_{ W_{ -1 - \delta_0 }^{3,p} ( M \setminus B_R ) } \le \epsilon, \quad \| \bar{ \pi }_k^{even} \|_{  W_{ -2 - \delta_0 }^{1,p} ( M \setminus B_R ) } \le \epsilon, \quad \mbox{ for } k \ge k_0.
$$
Moreover, the mass, the linear momentum, the center of mass, and the angular momentum of $ ( \bar{ g }_k, \bar{ \pi }_k)$ are within $ \epsilon $ of those of $ ( g, \pi )$.
\end{thm}

We first briefly describe Corvino and Schoen's construction of the approximating solutions $( \bar{ g }_k , \bar{ \pi }_k)$. 

%%%%%%%%%%%%%%%%%%%%%%%%Sketch of the Corvino--Schoen Density theorem%%%%%%%%%%%%%%%%%%%%%%%%%%%%%%%%%%%%%%%
\begin{proof}[Sketch of the proof of Theorem \ref{DenThm}]
Let $( \hat{ g }_k, \hat{ \pi }_k)$ be 2-tensors cut off from the original solutions $( g, \pi )$, $ \hat{ g }_k = \xi_k g + ( 1 - \xi_k ) \delta = \delta + \xi_k h$, $\hat{ \pi }_k = \xi_k \pi$ where $ \xi_k $ is a smooth cut-off function 
$$
		\xi_k (x) = \left\{\begin{array}{ll} 1 & \mbox{ when } |x| \le k \\
				\mbox{between $0$ and $1$} & \mbox{ when }  k \le |x| \le 2 k\\
 				0 & \mbox{ when } |x| \ge 2 k.
		\end{array} \right. 
$$
$\xi_k$  is chosen so that $ | D \xi_k | \le c/ k $ and $ | D^2 \xi_k | \le c/k^2$ for some constant $c$ independent of $k$. Then we let 
\begin{align*} 
		&\bar{ g }_k = u_k^4 \hat{ g }_k,\\
		& \bar{ \pi }_k = u_k^2 ( \hat{ \pi }_k + \mathcal{ L }_{\hat{ g }_k} X_k ).
\end{align*}
To simplify notation, we denote $\mathcal{L}_{\hat{g}_k}$ by $\mathcal{L} $, and we also drop the subindex $k$ when it is clear from context. In order to find $ u $ near $ 1 $ and $ X $ near $ 0 $ at infinity so that $ ( \bar{ g }, \bar{ \pi } ) $ is a vacuum initial data set, we need to solve the following systems for $u$ and $X$ from the constraint equations (\ref{PiCE})
\begin{eqnarray}\label{CE}
		&& \bar{ \mu } \equiv u^{ -5 } \bigg( -8 \Delta_{ \hat {g} } u + \Big( R( \hat{g} ) - | \hat{ \pi } + \mathcal{L}X |_{ \hat{ g } }^2 + \frac{1}{2} \big(Tr_{ \hat{ g } } ( \hat{ \pi } + \mathcal{ L } X  ) \big)^2 \Big) u\bigg) = 0\notag \\
		&& \big( \mbox{div}_{ \bar{ g } } ( \bar{ \pi } ) \big)_i = u^{-2} \bigg( \big( \mbox{div}_{ \hat{ g } } ( \hat{ \pi } + \mathcal{L} X )\big)_i + 4 u^{-1} \hat{ g }^{jk} u_j ( \hat{ \pi } + \mathcal{ L} X )_{ik} \notag\\
		&&\hskip 100pt - 2 u^{-1} u_i Tr_{ \hat{ g } } ( \hat{ \pi } + \mathcal{L} X )\bigg) = 0, \quad i = 1,2,3.  
\end{eqnarray}
Consider the map $T : (1,0) + W^{3,p}_{ -\delta} ( M )  \times W^{2,p}_{ -\delta} ( M )  \rightarrow W^{1,p}_{-2 -\delta} ( M )  \times W^{0,p}_{-2 -\delta} ( M ) $ defined by $ T( u , X ) = \big(\bar{ \mu }, \mbox{div}_{ \bar{ g } } ( \bar{ \pi } ) \big)$. It is known that $ DT_{ ( 1,0 )} $ is a Fredholm operator of index $0$. For the Fredholm operator of index $0$, the operator is injective if and only it is surjective. However, it is not clear whether $DT_{ (1,0)}$ is surjective. Corvino and Schoen enlarge the domain and utilize initial data sets of the form $ \big( u^4 \hat{ g } + h, u^2 ( \hat{ \pi } + \mathcal{ L } X) + q \big) $ with
$$
		\Phi \left( u^4 \hat{ g } + h, u^2 ( \hat{ \pi } + \mathcal{ L } X) + q \right) = ( 0 , 0 ), 
$$
where $ h $ and $ q $ are symmetric $ ( 0 , 2 )$-tensors with compact supports. Then they prove that the operator $D \Phi_{ ( \hat{g}, \hat{ \pi } ) }$ maps surjectively onto $ W^{1,p}_{-2 -\delta} ( M )  \times W^{0,p}_{-2 -\delta} ( M ) $ for $ p > 1 $ and $ \delta \in ( 0, 1 )$. 

%%%%%%%%%%%%%%%%%%%%%%Functional analysis argument%%%%%%%%%%%%%%%%%%%%%%%%%%%%%%%
Since $DT_{( 1, 0 )}$ is Fredholm of index $0$, we have 
\begin{eqnarray*}
		&& W^{3,p}_{-\delta} ( M )  \times W^{2,p}_{ -\delta} ( M )  = \mbox{Ker} \left( DT_{ ( 1 , 0 ) } \right) \oplus W_1,\\
		&& W^{1,p}_{-2 -\delta} ( M )  \times W^{0,p}_{-2 -\delta} ( M )  = \mbox{Range} \left( DT_{ ( 1 , 0 ) } \right) \oplus \mbox{span}\{ V_1, \dots, V_N\}
\end{eqnarray*}
where $ W_1 $ is an $N$-dimensional linear subspace and $\{ V_1, \cdots, V_N \}$ is a basis for the cokernel of $DT_{ (1,0) }$. Because $D \Phi_{ ( \hat{g}, \hat{ \pi } ) }$ is surjective, we can choose $\big\{ (h_1, q_1), \dots ,( h_N, q_N ) \big\}$ so that $D\Phi_{ ( \hat{ g }, \hat{ \pi } ) } ( h_i , q _i ) = V_i$. Supports of those $( h_i, q_i )$ may not be compact, but there exist $( \widetilde{ h}_i,  \widetilde{ q}_i )$ with compact supports close enough to $( h_i, q_i )$ so that $ \widetilde{V}_i= D \Phi_{ ( \hat{ g }, \hat{ \pi } ) }  ( \widetilde{ h}_i,  \widetilde{ q}_i ) $ still span a complementing subspace for $ \mbox{Range} \left( DT_{ ( 1 , 0 ) } \right)  $. 

Let $W_2 = \mbox{ span} \big\{ ( \widetilde{ h}_1,  \widetilde{ q}_1 ), \dots, ( \widetilde{ h}_N,  \widetilde{ q}_N ) \big\}$. $ W = W_1 \times W_2 $ is a Banach space inside $ W^{3,p}_{-\delta} ( M ) \times W^{2,p}_{ -\delta} ( M )  \times W^{3,p}_{ -\delta} ( M )  \times W^{1,p}_{-1-\delta} ( M )  $. Define the map $ \overline{T} $ from $ \big( (1,0), (0,0) \big) + W_1\times W_2 $ to $ W^{1,p}_{-2 -\delta} ( M )  \times W^{0,p}_{-2 -\delta} ( M )  $ by 
$$
		\overline{T}\big( (u,X) , ( h,q) \big) = \Phi( u^4 \hat{ g } + h , u^2( \hat{ \pi } + \mathcal{L} X) + q  ).
$$
$ D\overline{T}_{\left( (1 , 0 ), ( 0, 0 )\right)} $ is an isomorphism by construction. Hence the inverse function theorem asserts that $ \overline{ T}$ is an isomorphism from a fixed (independent of $k$) neighborhood of $ \left( (1,0) , ( 0 , 0 ) \right)$ to a fixed neighborhood of $\Phi( \hat{g}, \hat{ \pi })$. Because $(0,0)$ is contained in the image when $k$ large, there exists a unique $\left( ( u, X ), (h, q) \right)$ within that fixed neighborhood of $ \big( (1 , 0 ), ( 0, 0 )\big) $ such that $\Phi \left( u^4 \hat{ g } + h, u^2 ( \hat{ \pi } + \mathcal{ L } X) + q \right) = ( 0 , 0 )$ for $k$ large.
\end{proof}

Note that the supports of $h, q$ in the proof may not be uniformly bounded in $k$, but it is important in the proof of Theorem \ref{DenThm2} that $h,q$ have compact supports uniformly bounded in $k$. Therefore, we need to carefully choose the cokernel of $DT_{(1,0)}$ for $k$ large. This choice is described by the following lemma.

%%%%%%%%%%%%%%%%%%%%%%%Lemma to choose cokernel%%%%%%%%%%%%%%%%%%%%%%%%%%%%%%%%%%%%%%%%
\begin{lemma}
$V$ and $W$ are Banach spaces. Assume that $S_k : V \rightarrow W $ is a sequence of Fredholm operators and $ S_k$ converges (in the operator norm) to some Fredholm operator $ S' $. If $ W = {\rm Range } S' \oplus W' $ for some finitely dimensional closed subspace $W'$, we can choose the cokernel of $S_k$ inside the cokernel $W'$ of $S'$ for some $W_k \subset W'$, for $k$ large. 
 
\end{lemma}
\begin{proof}
Since $ S'$ is Fredholm, there exists $V'$ such that $V = \mbox{Ker} S' \oplus V'$. Consider for any bounded linear operator $S: V \rightarrow W$, the map
$$ 
		\tau_S : V' \oplus W' \rightarrow W
$$
given by $\tau_S ( v, w ) = S v + w$. Then $\tau_{ S_k} $ is an isomorphism for $ k $ large since $\tau_{ S' }$ is an isomorphism by construction and isomorphism is an open condition in the space of linear operators. We then have 
$$
		W = S_k( V' ) \oplus W'.
$$ 
Therefore, for any $v\in V$,  $ S_k ( v ) $ can be decomposed uniquely into $S_k(v') + S_k ( v) - S_k ( v')$ for some $v' \in V'$  and $ S_k ( v) - S_k ( v') \in W'$. Hence $v$ can be decomposed uniquely into $ v = v' + (v - v')$  as well, where $ S_k(v - v') \in W'$. Let $ U_k \equiv \{ u \in V : S_k(u) \in W'\}$. It is easy to see that $U_k$ is a closed space in $V$ and 
$$ 
		V = V' \oplus U_k.
$$ 
Note that $\mbox{Ker} S_k$ is a finite dimensional subspace in $U_k$, so we can write 
$$
		V = V' \oplus \mbox{Ker} S_k \oplus Z_k
$$ 
for some closed subspace $Z_k \subset U_k$. $ S_k(Z_k) $ is closed in $W'$, so there is $W_k \subset W' $ such that $ S_k(Z_k) \oplus W_k = W' $ and hence 
$$ 
		W = S_k( V' ) \oplus S_k(Z_k ) \oplus W_k = \mbox{Range} S_k \oplus W_k. 
$$ 
\end{proof}
%%%%%%%%%%%%%%%%%%%%%% Should be a picture here %%%%%%%%%%%%%%%%%%%%%%%%%%%%%%%%%%%%

\begin{cor}\label{CptSupp}
The supports of $h$ and $q$ in the proof of Theorem \ref{DenThm} can be chosen to be uniformly bounded for $k$ large. 
\end{cor}

Next, the key lemma (Lemma \ref{Key}) used to prove Theorem \ref{DenThm2} is an {\it a priori} type estimate for $2$nd order elliptic equations $P v = f $ which have a symmetric property at infinity. Roughly speaking, if we know $P $ and $f$ are even (or odd, respectively) then we hope solutions $v$ will be even (or odd, respectively) as well. However, this is not true generally because boundary values can affect solutions dramatically. For example, consider two harmonic functions $ u_1 , u_2 $ in $\mathbb{ R }^3 \setminus B_{R_0}$,
$$
		u_1 = \frac{ 1 }{ | x | }  \mbox{ and  } u_2 = \frac{ c \cdot x }{ | x |^3}. 
$$
Both $ | u_1 | $ and $| u_2 | $ tend to zero at infinity. However, $ u_1 $ is even and $ u_2 $ is odd. They are solutions for different boundary values on the inner boundary. Nevertheless, in the case that the boundary value is very small, we will show the symmetry of the solutions is not affected much in the region away from the boundary. Before we state the lemma, we will give some definitions of the operators we consider. This class of operators is discussed in detail in \cite{Bartnik}.

%%%%%%%%%%%%%%%%%%%%%%%%% Operators of Asymptotics %%%%%%%%%%%%%%%%%%%%%%%%%%%%%%%%%%%%%%%%%%%%%%%%%%%%
\begin{defi} \label{def:evenop}
Let $P$ defined as $ Pu = a^{ij}(x) \partial^2_{ij} u + b^i (x) \partial_i u + c (x) u$ be an elliptic operator. 
\begin{itemize}
\item[(1)]$ P $ is said to be asymptotic (at rate $ \tau $) to an elliptic operator $ \widetilde{ P } $, $ \widetilde{ P } u = \widetilde { a }^{ij} (x) \partial^2_{ ij } u + \widetilde{ b }^i ( x ) \partial_i u + \widetilde{ c }(x) u$, if there exist $ q \in  ( 3, \infty) $, $ \tau \ge 0 $, and a constant $ C $ such that over the region $ M \setminus B_{R_0}$,
\begin{align*}
		\| a^{ ij } - \widetilde {a }^{ij} \|_{ W^{ 1, q }_{ - \tau } } + \| b^i - \widetilde{ b }^i \|_{ W^{0,q }_{ -1 - \tau } } + \| c - \widetilde{ c } \|_{ W^{ 0, q/2 }_{ -2 - \tau } } \le C.  
\end{align*}
\item[(2)] $P^{ odd } $ is the odd part of the operator, defined on $ M \setminus B_{ R_0 }$ by
$$
		P^{ odd }  = ( a^{ ij } )^{ odd } \partial_{ij}^2  + ( b^i )^{even} \partial_i  + c^{odd }.
$$
\end{itemize}
\end{defi}
\begin{Remark}
We assume that $P: W^{s,p}_{-a} \rightarrow W^{s-2, p}_{-2-a}$ for $s = 2$ or $3$, $1 < p < q$ and a non-integer $a>0$, and define
\[
	\| P \|_{op; M \setminus B_{R}} = \sup \{ \| Pw \|_{W^{s-2, p}_{-2-a} (M)} : \| w \|_{W^{s,p}_a(M)} =1, \mbox{supp} w \subset M \setminus B_R\}.
\]
From Definition \ref{def:evenop}, we have
\begin{align*}
		\| P w - \widetilde{P} w \|_{W^{s-2, p}_{-2-a} (M)} \le&  \| a^{ ij } - \widetilde {a }^{ij} \|_{ W^{ 1, q }_{ - \tau } ( M \setminus B_R ) } \\
		&+ \| b^i - \widetilde{ b }^i \|_{ W^{0,q }_{ -1 - \tau } ( M \setminus B_R ) } + \| c - \widetilde{ c } \|_{ W^{ 0, q/2 }_{ -2 - \tau } ( M \setminus B_R ) }\rightarrow 0
\end{align*}
as $ R \rightarrow \infty$. Therefore, $\| P - \widetilde{P} \|_{op; M \setminus B_R} \rightarrow 0$ as $R \rightarrow 0$.
\end{Remark}
\begin{defi}
We say that a sequence of elliptic operators $ P_k $,	$P_k u = a_{(k)}^{ij}(x) \partial_{ij}^2 u + b_{(k)}^i (x) \partial_i u + c_{(k)} (x) u$, is asymptotic to $ \widetilde{ P }$ uniformly if given $ \epsilon $, there exist $R$ and $k_0$ such that
\begin{align*}
		&\| a^{ ij }_{ (k) }  - \widetilde {a }^{ij} \|_{ W^{ 1, q }_{ - \tau } ( M \setminus B_R )} + \| b^i_{ (k) }  - \widetilde{ b }^i \|_{ W^{0,q }_{ -1 - \tau }( M \setminus B_R ) }& \\
		&\hskip 120pt + \| c_{ (k) }  - \widetilde{ c } \|_{ W^{ 0, q/2 }_{ -2 - \tau }( M \setminus B_R ) } \le \epsilon  \mbox{ for all } k > k_0.&
\end{align*}
\end{defi}

%%%%%%%%%%%%%%%%%%%%%%%%% A prior estimate Lemma %%%%%%%%%%%%%%%%%%%%%%%%%%%%%%%%%%%%%%%%%%%%%
\begin{lemma}\label{Key}
Let $\Delta_{\delta}$ be the Euclidean Laplacian in $\mathbb{R}^3$. Let $\{ P_k \}$ be a sequence of  $2$nd order elliptic operators asymptotic to $\Delta_{\delta}$  uniformly, and for $s = 2$ or $3$, $1<p<q$, and  a non-integer number $a>0$,
$$
		P_k : W^{ s, p}_{ - a}(M) \rightarrow W^{s-2, p}_{ -2 - a }(M).
$$
Assume $\{v_k\} \subset W^{ s, p}_{ - a}(M)$, $\{ f_k \} \subset W^{s-2, p}_{ -3 - a }(M)$ are sequences of functions. We also assume $(v_k)^{odd} \in W_{ -1- a }^{s,p} ( M \setminus B_{ R_0} )$, $ \| (v_k)^{odd} \|_{ W_{ - a }^{s,p} ( M \setminus B_{ R_0} ) } \rightarrow 0 $ as $ k \rightarrow \infty $, and $(v_k)^{odd}$ satisfy $P_k (v_k)^{odd} = f_k$. Then there exist $R$ large and $k_0$ large such that 
$$
		\| (v_k)^{ odd } \|_{ W^{ s, p }_{ -1 - a } ( M \setminus B_{2 R}) } \le c \| f_k \|_{ W^{ s-2 ,p }_{ - 3 - a }( M \setminus B_R ) } + c\mbox{ for all } k > k_0,
$$
where $ c= c \big(s, p, a, \| P_k - \Delta_{\delta} \|_{op; M \setminus B_{R_0}} \big) , R = R\big( s, p, a, \| P_k - \Delta_{\delta} \|_{op; M\setminus B_{R_0}} \big) , k_0 = k_0 (R). $
\end{lemma}
%%%%%%%%%%%%%%%%%%%%%%%%%%%Proof of the A prior Estimate Lemma %%%%%%%%%%%%%%%%%%%%%%%%%%%%%%%%%%%%%%%%%%%%%%%%%%
\begin{proof}
Because $  (v_k)^{ odd } $ is not a global function, we multiply it by a smooth function $\phi$,
$$
		\phi (x) = \left\{\begin{array}{ll} 0 & \mbox{ when } |x| \le R \\
				\mbox{between $0$ and $1$} & \mbox{ when }  R \le |x| \le 2 R\\
 				1 & \mbox{ when } |x| \ge 2 R,
		\end{array} \right. 
$$
where $ \phi $ satisfies $ | D\phi | \le c/ R, | D^2 \phi | \le c /R^2$ for some constant $c$ independent of $R$, where the real number $R$ will be chosen later. We have $ \phi (v_k)^{ odd } \in W^{ s, p }_{ -1 - a }( M ) $ and 
\begin{align*}
		P_k  \big( \phi (v_k)^{ odd } \big) =& \phi f_k + a_{ (k) }^{ ij} \Big( \frac{ \partial^2 \phi }{ \partial x^i \partial x^j} \Big) (v_k)^{ odd } \\
		&+ 2 a_{ (k) }^{ij} \frac{ \partial \phi }{ \partial x^i }  \frac{ \partial (v_k)^{ odd } }{ \partial x^j } + b^i_{ ( k ) } \Big( \frac{ \partial \phi }{ \partial x^i} \Big) (v_k)^{ odd }.
\end{align*}
By \cite[Theorem 1.7]{Bartnik}, there exists a constant $ c_1 = c_1(s, p,a )$ such that 
$$
		\| \phi (v_k)^{ odd } \|_{ W^{s,p}_{ -1 - a}( M ) } \le c_1 \| \Delta_{\delta} ( \phi (v_k)^{ odd }  ) \|_{ W^{ s - 2,p}_{ -3 - a}( M ) }.
$$
We then consider the difference between $\Delta_{\delta}$ and $P_k$ by viewing them as the operators from $W^{s,p}_{-1-a}(M)$ to $W^{s-2, p}_{-3-a}(M)$. By the remark after Definition \ref{def:evenop}, the operator norm $\| P_k - \Delta_{ \delta }\|_{op; M\setminus B_R}$ tends to zero as $R \rightarrow \infty$ uniformly in $k$ because $P_k$ is asymptotic to $\Delta_{\delta}$ uniformly. 
%defined by
%\begin{align*}
%		\| P_k - \Delta_{ \delta }\|_{op; M\setminus B_R} = \sup \Big\{ \| (P_k - \Delta_{\delta}) w \|_{ W^{s-2, p}_{ -3 - \delta } (M) }: \| w \|_{ W^{ s,p }_{ -1 - a } ( M ) }  = 1, \\
%		\mbox{ supp}w \subset M \setminus B_R \Big\}
%\end{align*}
%tends to zero as $R\rightarrow \infty$ uniformly in $k$ because $P_k$ is asymptotic to $\Delta_{ \delta }$ uniformly. More precisely, for $\omega$ as above,
%\begin{eqnarray*}
%		\| ( P_k - \Delta_{\delta} ) w\|_{ W^{ s-2,p }_{ -3 - a } ( M ) } &\le& \sup_{ |x| \ge R } \Big\{ \big| a_{ (k) }^{ ij} - \delta^{ij} \big|  \Big\} \| D^2 w\|_{ W^{ s-2,p }_{ -3 - a } ( M ) } \\
%		&&+ c_2 \| b \|_{ W^{ 0 , q }_{ -1  } ( M \setminus B_R ) } \| D w \|_{ W^{ s-1,p }_{ -2 - a } ( M ) } \\
%		&&+ c_2 \| c \|_{ W^{ 0, q/2 }_{ -2  } ( M \setminus B_R ) } \| w \|_{ W^{ s,p }_{ -1 - a } ( M ) } \\
%		& \rightarrow&  0 \quad \mbox{ as } R \rightarrow \infty \mbox{ uniformly in } k.
%\end{eqnarray*}
Therefore, 
\begin{align*}
		&\| \phi (v_k)^{odd} \|_{ W^{ s ,p }_{ -1 - a } ( M ) }\\
		 &\le c_1 \| ( P_k - \Delta_{\delta} )( \phi (v_k)^{odd}) \|_{ W^{ s - 2 ,p }_{ -3 - a } ( M ) } + c_1 \|  P_k ( \phi (v_k)^{odd}) \|_{ W^{ s - 2 ,p }_{ -3 - a } ( M ) }\\
		& \le  c_1   \|  P_k - \Delta_{\delta}  \|_{op; M \setminus B_R }\| \phi (v_k)^{odd} \|_{ W^{ s,p }_{ -1 - a } ( M ) } + c_1 \|  P_k ( \phi (v_k)^{odd}) \|_{ W^{ s-2 ,p }_{ -3 - a } ( M ) }.
\end{align*}
Choose $ R $ such that $  c_1   \|  P_k  - \Delta_{\delta}  \|_{op; M \setminus B_R }  \le \frac{1}{2}$ and absorb that term to the left hand side. Then
\begin{align*}
		\| \phi (v_k)^{odd} \|_{ W^{ s,p }_{ -1 - a } ( M ) } \le & 2 c_1 \|  P_k ( \phi (v_k)^{odd}) \|_{ W^{ s-2,p }_{ -3 - a } ( M ) } \\
		 \le & c_2 \| f_k \|_{ W^{ s-2,p }_{ -3 - a } ( M \setminus B_R) } + c_2 \| ( D^2 \phi ) (v_k)^{ odd }\|_{ W^{ s-2,p }_{ -3 - a } ( A_R ) } \\
		& + c_2  \|  D \phi \cdot D (v_k)^{ odd }\|_{ W^{ s-2,p }_{ -3 - a } ( A_R ) } + c_2 \| ( D \phi ) (v_k)^{odd} \|_{ W^{ s-2,p }_{ -2 - a } ( A_R ) } \\
		 \le & c_2 \| f_k \|_{ W^{ s-2,p }_{ -3 - a } ( M \setminus B_R) } + c_3 \| (v_k)^{ odd }\|_{ W^{ s-1,p }_{ -1 - a } ( A_R ) },
\end{align*}
where $A_R = \{ x: R \le |x| \le 2R\}$. Because $\| (v_k)^{ odd }\|_{ W^{ s-1,p }_{ - a } ( M \setminus B_{R_0}) } \rightarrow 0 $ as $ k \rightarrow \infty$, for $ \epsilon = 1/(2  R) $, there exists $ k_0 $ such that for all $ k > k_0$
$$
		\| (v_k)^{ odd } \|_{ W^{ s-1 ,p }_{ -1 - a } ( A_R ) } \le 2 R \| (v_k)^{ odd } \|_{ W^{ s-1 ,p }_{ - a } ( A_R ) } \le 2 R\epsilon = 1.
$$
As a result, we have the Schauder-type estimate
$$
		\| (v_k)^{ odd } \|_{ W^{ s, p }_{ -1 - a } ( M \setminus 2 R) } \le c_2 \| f_k \|_{ W^{ s-2 ,p }_{ - 3 - a }( M \setminus B_R ) } + c_3\mbox{ for all } k > k_0.
$$ 
\end{proof}

%%%%%%%%%%%%%%%%%%%%%%%%% Proof of the Density Theorem %%%%%%%%%%%%%%%%%%%%%%%%%%%%%%%%%%%%%%%%%%%%%%%%%%%%%%%%%%%%%%%
\begin{proof}[Proof of Theorem \ref{DenThm2}]
\quad

%%%%%%%%%%%%%%%%%%%%%%%%% Estimate on u^odd and X^{odd} %%%%%%%%%%%%%%%%%%%%%%%%%%%%%%%%%%%%%%%%%%%%%%%%%%%%%%%%%%%%%%%
\noindent\textbf{\small 1. Estimates on $u^{odd} $ and $ ( X^i )^{ odd }$.} We construct $ ( \bar{g}_k , \bar { \pi }_k ) $ as in Theorem \ref{DenThm} in the form 
\[
	\bar{ g }_k = u_k^4 \hat{g}_k + h_k , \quad \bar{ \pi }_k = u_k^2 ( \hat{ \pi }_k + \mathcal{L}_{\hat{g}_k} X_k ) + q_k.
\] 
Recall that $k$ is the radius of which we cut off the original data. Again, we drop the subindex $k$ when it is clear from context. By Theorem \ref{DenThm}, $ u $ and $X$ exist and satisfy the system of the constraint equations (\ref{CE}). From the constraint equation (\ref{CE}) for $u$ in $ M \setminus B_{ R_0 }$, we have
\begin{eqnarray*}
		0 &=& \Delta_{ \hat{ g } } u - \frac{1}{8} \Big( R( \hat{ g } ) - | \hat{ \pi  } + \mathcal{ L }X |_{ \hat{ g } }^2 + \frac{ 1 }{ 2 } \big( \mbox{Tr}_{ \hat{ g } } \hat{ \pi } + \mathcal{ L }X \big)^2\Big) u\\
		   &=& \hat{ g }^{ij} \frac{ \partial^2 u }{ \partial x^i \partial x^j} + \sqrt{ \hat{g} }^{ -1 } \Big( \frac{ \partial }{ \partial x^i} \hat{ g }^{ ij } \sqrt{ \hat{ g } }\Big) \frac{ \partial u }{ \partial x^j} \\
		   &&- \frac{1}{8} \Big( R( \hat{g} ) - | \hat{ \pi } + \mathcal{L}X |_{ \hat{ g } }^2 + \frac{1}{2} \big(\mbox{Tr}_{ \hat{ g } } ( \hat{ \pi } + \mathcal{ L } X  ) \big)^2 \Big) u \\
		   & \equiv & P_1 u.
\end{eqnarray*}
On $M \setminus B_{ R_0 }$, $ u^{ odd }$ satisfies the following equation, 
\begin{eqnarray*}
		P_1 u^{ odd } & = & ( P_1 u )(x) - ( P_1 u)(-x) - P^{ odd }_1\big( u(-x) \big)\\
		& = & \frac{ 1 }{ 8 } \left( R( \hat{ g } )- | \hat{ \pi } + \mathcal{L}X |_{ \hat{ g } }^2 + \frac{1}{2} \big(\mbox{Tr}_{ \hat{ g } } ( \hat{ \pi } + \mathcal{ L } X  )   \right)^{ odd } u( -x ) \\
		& & - \left( \hat{ g }^{ ij }(x) - \hat{ g }^{ij}( - x ) \right) \frac{ \partial^2 }{ \partial x^i \partial x^j} u(-x) \\
		& & - \left(  \sqrt{ \hat{g} }^{ -1 }( x ) \Big( \frac{ \partial }{ \partial x^i} \hat{ g }^{ ij } ( x ) \sqrt{ \hat{ g } }( x ) \Big) \right.\\
		&& \qquad \qquad \left.- \sqrt{ \hat{g} }^{ -1 }( - x ) \Big( \frac{ \partial }{ \partial x^i} \hat{ g }^{ ij } ( - x ) \sqrt{ \hat{ g } }( - x ) \Big)\right) \frac{ \partial }{ \partial x^j}  u( -x )\\
		&=& f_1 + f_2,
\end{eqnarray*}
where 
$$
		f_1= \frac{ 1 }{ 8 } \left( - | \hat{ \pi } + \mathcal{L}X |_{ \hat{ g } }^2 + \frac{1}{2} \mbox{Tr}_{ \hat{ g } } ( \hat{ \pi } + \mathcal{ L } X  )   \right)^{ odd } u( -x ), \quad f_2 = P_1 u^{ odd } - f_1.
$$
$f_1$ contains the terms involving $X$. We will use a bootstrap argument to improve its decay rate. $f_2$ contains the terms which have expected good decay rate already, such as $g^{odd}, \pi^{even}$. A direct calculation tells us 
$$
		\| f_1 \|_{ W^{ 1,p}_{ -2 - 2\delta } ( M \setminus B_{ R_0 }) } \le c, \quad \| f_2 \|_{ W^{ 1,p}_{ -3 -  \delta } ( M \setminus B_{ R_0 }) } \le c.
$$
We emphasize that through out this proof, $ c $ is a constant independent of $k$. 

From the constraint equations (\ref{CE}) for $X$, we have
\begin{align*}
		 &\big( \mbox{div}_{ \hat{ g } } ( \mathcal{ L }X) \big)_i  + \mbox{div}_{ \hat{ g } } ( \hat{ \pi } )_i + 4 u^{ -1 } \hat{ g }^{ kj } u_j \big( \hat{ \pi } + \mathcal{ L } X \big)_{ik} \\
		 &\qquad  \qquad - 2 u^{ -1 } u_i \mbox{Tr}_{ \hat{ g } } (\hat{ \pi } + \mathcal{ L }X ) = 0 .
\end{align*}
If we compute the first term in local coordinates, we have
\begin{align*}
		\big( \mbox{div}_{ \hat{ g } } ( \mathcal{ L }X) \big)_i &= \big( \Delta_{ \hat{ g }} X \big)_i + \hat{Ric}_{ik} X^k \\
		&= \hat{ g }^{kl} \frac{ \partial^2 }{ \partial x^k \partial x^l } X_i - \hat{g}^{ kl } \frac{ \partial }{ \partial x^k} ( X^p \Gamma_{li}^q \hat{ g }_{ pq } )- \hat{ g }^{ kl } X^p \Gamma^q_{ lp } \Gamma^r_{ ki }\hat{ g }_{ qr } + \hat{Ric}_{ik} X^k.
\end{align*}
We define $ P_2 X_i \equiv \hat{ g }^{kl} \frac{ \partial^2 }{ \partial x^k \partial x^l } X_i$ and then $ P_2 X_i = F_i$, where $ F_i $ contains the remainder terms from above identities,
\begin{eqnarray*}
		F_i & = & \hat{g}^{ kl } \frac{ \partial }{ \partial x^k} ( X^p \Gamma_{li}^q \hat{ g }_{ pq } ) + \hat{ g }^{ kl } X^p \Gamma^q_{ lp } \Gamma^r_{ ki }\hat{ g }_{ qr } - \hat{Ric}_{ik} X^k  \\
		&& - \mbox{div}_{ \hat{ g } } ( \hat{ \pi } )_i - 4 u^{ -1 } \hat{ g }^{ kj } u_j \big( \hat{ \pi } - \mathcal{ L } X \big)_{ik} + 2 u^{ -1 } u_i \mbox{Tr}_{ \hat{ g } } (\hat{ \pi } + \mathcal{ L }X ).
\end{eqnarray*}
Then we have 
\begin{align*}
		 P_2 \left( X_i\right)^{odd} &= \left(F_i \right)^{odd} - \left( \hat{ g }^{kl}(x) \frac{ \partial^2 }{ \partial x^k \partial x^l } -\hat{ g }^{kl}(-x) \frac{ \partial^2 }{ \partial x^k \partial x^l } \right)X_i(-x)\\
		  &= a_i + b_i,
\end{align*} 
where
$$
		a_i = -\left( \mbox{div}_{ \hat{ g } } ( \hat{ \pi } )_i\right)^{odd}, \quad b_i= P_2 \left( X_i\right)^{odd} - a_i, 
$$
where $b_i$ contains $u^{odd}$ and $X^{odd}$ and we will bootstrap to improve its decay rate. A straightforward calculation gives us 
$$
				\| a_i \|_{ W^{ 0,p}_{ -3 - \delta } ( M \setminus B_{ R_0 }) } \le c, \quad \| b_i \|_{ W^{ 0,p}_{ -2 - 2  \delta } ( M \setminus B_{ R_0 }) } \le c.
$$
From the above, we derive the system
\begin{eqnarray*}
		&& P_1 u^{odd} = f_1 + f_2,\\
		&& P_2 (X_i)^{ odd } = a_i + b_i, \quad i = 1,2,3.
\end{eqnarray*}
We apply Lemma \ref{Key} (with $a=2\delta -1 \le \delta$) for each equation because $P_1$ and $P_2$ are obviously asymptotic to $\Delta_{\delta}$ and 
\begin{align*}
		&\| (u_k)^{odd}\|_{W^{3,p}_{-a}(M \setminus B_{R_0})} \le \| (u_k)^{odd} \|_{W^{3,p}_{- \delta} (M\setminus B_{R_0})} \le 2 \| u_k - 1\|_{W^{3,p}_{-\delta} (M)}  \rightarrow 0,\\
		& \| (X_i)_k^{odd}\|_{W^{2,p}_{-\delta} (M \setminus B_{R_0})} \le \| (X_i)_k^{odd} \|_{W^{2,p}_{- \delta} (M\setminus B_{R_0})} \le 2 \| (X_i)_k\|_{W^{2,p}_{-\delta} (M)}  \rightarrow 0\quad \mbox{ as } k \rightarrow \infty.
\end{align*}
Because $-1-a = 2\delta$, by Lemma \ref{Key}, there exist $R_1$ and $k_1$ such that for all $k > k_1$, 
$$
		\| (u_k)^{odd} \|_{ W^{3,p}_{  - 2 \delta } ( M \setminus B_{R_1} )} \le c, \quad \| (X_i)_k^{odd} \|_{ W^{2,p}_{  - 2 \delta } ( M \setminus B_{R_1} )} \le c.
$$
Once we derive these estimates, the decay rates for $f_1$ and  $b_i$ are improved. The bootstrap argument allows us to conclude that for some $R_2 \ge R_1$, $k_2 \ge k_1$,  
$$
		\| (u_k)^{odd} \|_{ W^{3,p}_{ - 1 -  \delta } ( M \setminus B_{R_2} )} \le c, \quad \| (X_i)_k^{odd} \|_{ W^{2,p}_{ - 1 -  \delta } ( M \setminus B_{R_2} )} \le c \mbox{ for all } k > k_2.
$$
Therefore, for any given $\epsilon$ and $\delta_0 \in ( 0, \delta)$, there exist $R$ and $k_0$ so that for all $k > k_0$, 
$$
		\| (u_k)^{odd} \|_{ W^{3,p}_{ - 1 -  \delta_0 } ( M \setminus B_R )} \le CR^{\delta_0 - \delta} \le  \epsilon, \quad \| (X_i)_k^{odd} \|_{ W^{2,p}_{ - 1 -  \delta_0 } ( M \setminus B_R )} \le \epsilon.
$$
Furthermore, by Corollary \ref{CptSupp}, the supports of $(h_k, q_k)$ are uniformly bounded in $k$. Hence we have 
$$
		\| \bar{g}_k^{ odd }\|_{ W^{3,p}_{-1 - \delta_0 } ( M \setminus B_R ) } \le \epsilon.
$$
Similarly, we have the estimate for $\bar{\tau}$.
%%%%%%%%%%%%%%%%%%%%%%%%Convergence of the Center of mass and Angular Momentum%%%%%%%%%%%%%%%%%%%%%%%%%%%%%%%%%%%%%%%%%%%%%%%%%%%%%%%%%%%%

\noindent\textbf{\small 2. Convergence of the Center of Mass and the Angular Momentum.} To prove the center of mass and the angular momentum of $( \bar{ g}, \bar{\pi})$ converge to those of $ ( g, \pi) $, the same idea of proving convergence of the mass and the linear momentum in \cite{CS} is employed. 

\begin{eqnarray*}
		&& | C_I^{\alpha}(g) - C_I^{\alpha}( \bar{g}) | \le \left| C_I^{ \alpha } (g) - \int_{ |x| = r } \big( R_{ij} - \frac{1}{2} R_g g_{ij} \big) Y_{(\alpha)}^i\nu_g^j \,d \sigma_g\right| \\
		 &&\qquad + \left| \int_{ |x| = r } \big( R_{ij} - \frac{1}{2} R_g g_{ij} \big) Y_{(\alpha)}^i \nu_g^j\, d \sigma_g - \int_{ |x| = r } \big( \bar{R}_{ij} - \frac{1}{2} \bar{R} \bar{g}_{ij} \big) Y_{(\alpha)}^i \nu_{ \bar{g}}^j \,d \sigma_{ \bar{g} }\right| \\
		&& \qquad + \left| C_I^{ \alpha }(\bar{g})  - \int_{ |x| = r } \big( \bar{R}_{ij} - \frac{1}{2} \bar{R} \bar{g}_{ij} \big)  Y_{(\alpha)}^i \nu_{ \bar{g}}^j \,d \sigma_{ \bar{g} }\right| .
\end{eqnarray*} 
The first and the third terms can be written as integrals over $\{ |x| \ge r\}$ by the divergence theorem. We will also use the fact that $\{ |x| \ge r\}$ is centrically symmetric, i.e. $x, -x \in \{ |x| \ge r\}$, to estimate those integrals over $\{ |x| \ge r\}$. For the first integral,
\begin{eqnarray}
		&& C_I^{\alpha} (g) - \int_{ |x| = r } \big( R_{ij} - \frac{1}{2} R_g g_{ij} \big)  Y_{(\alpha)}^i \nu_g^j \,d \sigma_g \notag\\
		&=& \int_{ |x| \ge r } \left[ x^{\alpha} R_g - 2 x^{\alpha} (g^{ij} - \delta^{ij} ) \big( R_{ij} - \frac{1}{2} R_g g_{ij} \big) \right.\notag\\
		&&\left.+  \big( R_{ij} - \frac{1}{2} R_g g_{ij} \big)  Y_{(\alpha)}^i  \Gamma_{kl}^k g^{lj}\right] \, d \mbox{vol}_g. \label{Ann}
\end{eqnarray}
Let $H =  \{ x\in M \setminus B_r:  x^1 \ge  0 \}$ be the half space. Then we can rewrite the above integral as follows:
\begin{align*}
		 &\int_H x^{\alpha} \big( R_g^{odd}\sqrt{g} + R_g(-x) \sqrt{g}^{odd} \big) \,dx\\
		 & - \int_H   2 x^{\alpha} ( g^{ij} )^{odd} \big( R_{ij}- \frac{1}{2} R_g g_{ij} \big) \sqrt{g}  \,dx\\
		 &- \int_H 2 x^{\alpha} ( g^{ij}- \delta^{ij})\left[ \Big( R_{ij}^{odd} -  \frac{1}{2} R_g^{ odd } g_{ij} -  \frac{1}{2} R_g(-x) g_{ij}^{odd}\Big)\sqrt{ g }\right.\\
		 &\left. + \big( R_{ij}- \frac{1}{2} R_g g_{ij}  \big)\sqrt{g}^{odd} \right] \,d x \\
		 & + \int_H  Y_{(\alpha)}^i  \Big(R_{ij}^{odd} -  \frac{1}{2} R_g^{ odd } g_{ij} -  \frac{1}{2} R_g g_{ij}^{odd} \Big)\Gamma_{kj}^i \sqrt{ g} \,dx\\
		& + \int_H Y_{(\alpha)}^i  \big( R_{ij}(-x) - \frac{1}{2} R_g( -x ) g_{ij}(-x)\big)\left[ ( { \Gamma_{kj}^i })^{even} \sqrt{ g} + \Gamma_{kj}^i( - x)\sqrt{ g}^{odd} \Big) \right]\,dx.
\end{align*}
We substitute $R_g$ in the first integral by using the constraint equation $ R_g = - \frac{1}{2} \big( \mbox{Tr}_{g} \pi \big)^2 + | \pi |_g^2$. We then bound above integrals symbolically and use the H\"{o}lder inequality:
\begin{align} \label{Annulusg}
		& c_1 \int_H |x| \Big( | \pi^{even} | | \pi| + | \pi |^2 | g^{ odd } | +  | g^{odd} | | D^2 g | + | g - \delta | | D^2( g^{ odd } ) | \Big) \,dx \notag\\
		& + c_1\int_H |x|^2 \Big( | D^2  ( g^{odd}) | | Dg |  + | D^2 g| | D ( g^{odd}) | + | D^2 g| | Dg| | g^{odd} |\Big) \,dx\notag \\
		&\le  c_2  \Big( \| \pi^{ even } \|_{ W^{1,p}_{ -2 - \delta } } \| \pi \|_{ W^{1,p}_{ -1- \delta} } +  \| \pi \|^2_{ W^{1,p}_{ -1- \delta} }\|g^{odd}\|_{ W^{2,p}_{ -1- \delta} } \notag\\
		& \hskip 2in + \| g^{ odd } \|_{ W^{2,p}_{ -1- \delta} }\| g \|_{ W^{2,p}_{ - \delta}}\Big) r^{ 1-2 \delta},
\end{align}
where $ c_1$ and  $c_2 $ are constants independent of $g$, $\pi$, and $r$. The weighted Sobolev norms above are over the region $\{ |x| \ge r\}$. Similarly, 
\begin{eqnarray} \label{AnnulusgBar}
		 &&\left| C_I^{ \alpha }(\bar{g})  - \int_{ |x| = r } \big( \bar{R}_{ij} - \frac{1}{2} \bar{R} \bar{g}_{ij} \big)  Y_{(\alpha)}^i \nu_{ \bar{g}}^j \,d \sigma_{ \bar{g} }\right| \notag\\
		& \le& c_2 \Big( \| \bar{\pi}^{ even } \|_{ W^{1,p}_{ -2 - \delta } } \| \bar{\pi} \|_{ W^{1,p}_{ -1- \delta} } +  \| \bar{\pi} \|^2_{ W^{1,p}_{ -1- \delta} }\|\bar{g}^{odd}\|_{ W^{2,p}_{ -1- \delta} }  \notag \\
		&& \qquad + \| \bar{g}^{ odd } \|_{ W^{2,p}_{ -1- \delta} }\| \bar{g} \|_{ W^{2,p}_{ - \delta}}\Big) r^{ 1-2 \delta} .
\end{eqnarray}

For the surface integral, we can assume $ k\gg r$ (recall $k$ is the radius of which we cut off the original data) and then $\bar{g} = u^4 g$ on $\{ |x| < 2r\}$.  Hence we have 
\begin{align*}
		&\int_{ |x| = r } \big( R_{ij} - \frac{1}{2} R_g g_{ij} \big)  Y_{(\alpha)}^i \nu_g^j \,d \sigma_g - \int_{ |x| = r } \big( \bar{R}_{ij} - \frac{1}{2} \bar{R} \bar{g}_{ij} \big)  Y_{(\alpha)}^i \nu_{ \bar{g}}^j \,d \sigma_{ \bar{g} }\\
		&=  \int_{ |x| = r } \Big( ( R_{ ij } - u^2 \bar{ R }_{ ij }) - \frac{1}{2} (R_g - u^6 \bar{R} ) g_{ij} \Big) Y_{(\alpha)}^i  \nu_g^j \,d \sigma_g\\
		&= \int_{ |x| = r }  \Big( ( 1 - u^2 ) R_{ ij } - \frac{1}{2} ( 1 - u^6 ) R_g g_{ij}\Big) Y_{(\alpha)}^i  \nu_g^j \,d \sigma_g \\
		& \quad + \int_{ |x| = r } \Big( u^2 ( R_{ij} - \bar{R}_{ij} ) - \frac{1}{2} u^6 ( R_g - \bar{ R } )g_{ij} \Big) Y_{(\alpha)}^i  \nu_g^j\, d \sigma_g\\
		&\le c_6 \max_{ |x| = r} \Big( | 1 - u | |x|^{ -\delta} \Big) \max_{ |x| = r} \Big( | D^2 g | | x |^{ 2 + \delta } \Big) r^{ 2 -2 \delta } \\
		&\quad + c_6 \max_{ |x| = r } \big| D^2 ( g - \bar{g} ) \big| |x|^{ 2 + \delta } r^{ 2 - \delta }. 
\end{align*}
Recall the Sobolev inequality for weighted Sobolev spaces (see \cite{Bartnik}). For $ n - sp < 0$ where $n = \dim M = 3 $ , we have $\| w \|_{ W^{ 0, \infty}_{ -\delta} (M) } \le c \| w\|_{W_{-\delta}^{ s, p } (M)} $. Therefore, for $p> 3$, 
\begin{align*}
		\max_{ | x | = r} | D^2 ( g - \bar{g})  | |x|^{2+\delta} & \le  \max_{ M \setminus B_r} | D^2 (g - \bar{g} ) | |x|^{ 2 + \delta}\\
		& = \| D^2  ( g - \bar{g})  \|_{W^{0, \infty}_{ -2 - \delta} ( M \setminus B_r) } \le c \| D^2( g - \bar{g})  \|_{W^{1,p}_{- 2  -\delta} (M \setminus B_r) }\\
		 &\le c \|   g - \bar{g}  \|_{W^{3,p}_{  -\delta} (M \setminus B_r) }. 
\end{align*}
Similarly, using the Sobolev inequality on the other term, we derive
\begin{align}\label{SurfaceInt}
		&\int_{ |x| = r } \big( R_{ij} - \frac{1}{2} R_g g_{ij} \big)  Y_{(\alpha)}^i \nu_g^j\, d \sigma_g - \int_{ |x| = r } \big( \bar{R}_{ij} - \frac{1}{2} \bar{R} \bar{g}_{ij} \big)  Y_{(\alpha)}^i \nu_{ \bar{g}}^j \,d \sigma_{ \bar{g} } \notag\\
		&\le c_7 \Big( \| 1 - u\|_{W^{ 3,p}_{ -\delta }( M \setminus B_r) } + \| g - \bar{g} \|_{ W^{ 3, p}_{-\delta}( M \setminus B_r) }\Big) r^{ 2 - \delta }.
\end{align}
To conclude the argument, for given $\epsilon$, let $R_1$ be a constant so that the right hand side of \eqref{Annulusg} is bounded by $\epsilon /16$. For this $\epsilon$, there exist $R_2$ and $k\ge k_0$ so that $\| 1 - u_k\|_{W^{ 3,p}_{ -\delta }( M  ) },  \| g - \bar{g}_k \|_{ W^{ 3, p}_{-\delta}(M) }, \| \pi - \bar{\pi}_k\|_{W^{ 1,p}_{ -1 -\delta }( M  ) }, \| \bar{\pi}^{even}_k \|_{ W^{ 1, p}_{-2-\delta}(M \setminus B_{R_2}) }$ and $ \| \bar{g}^{odd}_k \|_{ W^{ 3, p}_{ -1 -\delta}(M \setminus B_{R_2}) } $ are small, for all $k \ge k_0$. Using estimates \eqref{Annulusg}, \eqref{AnnulusgBar}, and \eqref{SurfaceInt}, we can conclude that, for all $k \ge k_0$,
%let $R_1$ be a constant so that in (\ref{Annulusg}),
%\begin{align*}
%	\max\Big\{ &c_2  \Big( \| \pi^{ even } \|_{ W^{1,p}_{ -2 - \delta } } \| \pi \|_{ W^{1,p}_{ -1- \delta} } +  \| \pi \|^2_{ W^{1,p}_{ -1- \delta} }\|g^{odd}\|_{ W^{2,p}_{ -1- \delta} } + \| g^{ odd } \|_{ W^{2,p}_{ -1- \delta} }\| g \|_{ W^{2,p}_{ - \delta}}\Big) R_1^{ 1-2 \delta},\\
%	&c_2  \Big(   \| \pi \|_{ W^{1,p}_{ -1- \delta} } +  \| \pi \|^2_{ W^{1,p}_{ -1- \delta} } + \| g^{ odd } \|_{ W^{2,p}_{ -1- \delta} }\Big) R_1^{ 1-2 \delta}  \Big\}\le  \frac{ \epsilon }{ 16 }.
%\end{align*}
%For this $\epsilon$, there exist $R_2 \ge R_1$ and $k_0$ so that 
%\begin{eqnarray*}
%		&&\| 1 - u_k\|_{W^{ 3,p}_{ -\delta }( M  ) }  \le \frac{ \epsilon }{ 16 c_7 R_2^{ 2 - \delta} } \, , \quad  \| g - \bar{g}_k \|_{ W^{ 3, p}_{-\delta}(M) } \le \frac{ \epsilon }{ 16 c_7 R_2^{ 2 - \delta} },\\
%		&&\| \pi - \bar{\pi}_k\|_{W^{ 1,p}_{ -1 -\delta }( M  ) }  \le \frac{ \epsilon }{ 16 c_2 R_2^{ 2 - \delta} } \, , \quad \| \bar{g}^{odd}_k \|_{ W^{ 3, p}_{ -1 -\delta}(M \setminus B_{R_2}) } \le \epsilon \, ,\\
%		&&  \| \bar{\pi}^{even}_k \|_{ W^{ 1, p}_{-2-\delta}(M \setminus B_{R_2}) } \le \epsilon \quad \mbox{ for all } k \ge k_0.
%\end{eqnarray*}
%As a result, using the estimates (\ref{Annulusg}) (\ref{AnnulusgBar})  (\ref{SurfaceInt}) and try $r$ in those estimates equal to $R_2$, we can conclude for all $k \ge k_0 $,
\begin{eqnarray*}
		&& | C_I^{\alpha}(g) - C_I^{\alpha}( \bar{g}_k) | <  \epsilon.
\end{eqnarray*}

To prove that angular momentum of $ \bar{g} $ is close to that of $g$, we notice that 
\begin{align*}
		& J_{\alpha} - \int_{ |x| = r } \pi_{jk}  Z_{(\alpha)}^j \nu_g^k \, d \sigma_g = \int_{ |x| \ge r } div_g \big( \pi_{jk}  Z_{(\alpha)}^j  \big) \, d {\rm vol}_g \\
		 &= \int_{ |x| \ge r }  \pi_{ jk  } g^{kl}\Big( \frac{ \partial }{ \partial x^l }Z_{(\alpha)}^j + Z_{(\alpha)}^{m} \Gamma_{ ml }^j\Big)\, d {\rm vol}_g \\
		 &= \int_{ |x| \ge r } \pi_{ jk }  \frac{ \partial }{ \partial x^k } Z_{(\alpha)}^j + \pi_{ jk  }( g^{kl}- \delta^{kl} )\frac{ \partial }{ \partial x^l } Z_{(\alpha)}^j + \pi_{jk} g^{kl} Z_{(\alpha)}^{m} \Gamma_{ ml }^j\, d {\rm vol}_g.
\end{align*}
The first term is zero since $ \pi_{jk}$ is a symmetric tensor and $ Z^{\alpha}$ is a Killing vector field. The other terms after integration are bounded by $C r^{1 - 2 \delta}$. Then the rest of the argument works the same as the case of center of mass.
\end{proof}

In the proof of convergence of the center of mass and the angular momentum, we showed that the limits at infinity can be approximated by surface integrals at the finite radius and  the difference for surface integrals at the finite radius is arbitrarily small. Since we will use this argument several times through this paper, we formulate the argument into the following lemma.

%%%%%%%%%%%%%%%%%%%%%%%%%%%%% Convergence Lemma %%%%%%%%%%%%%%%%%%%%%%%%%%%%%%%%%%%%%%%%%%%%%%%%
\begin{lemma} \label{Conv}
Let $F(g)$ be a vector field depending on $x, g_{ij}, D g_{ij}, D^2 g_{ij} $ smoothly defined in $M \setminus B_{R_0}$. Let $a, b >0 $, and let $r$ be the radius. Assume that the following estimates hold for $k > r$, and for $g, \bar{g}_k$ satisfying the assumptions in Theorem {\ref{DenThm2}} (Density Theorem).
\begin{itemize}
\item[(1) ] 
\begin{align} \label{Annulus}
		&\bigg| \int_{ |x| \ge r } {\rm div}_g F(g) \,d{\rm vol}_g \bigg| \le c \bigg(  \| \pi^{ even } \|_{ W^{1,p}_{ -2 - \delta } ( M \setminus B_r )} \| \pi \|_{ W^{1,p}_{ -1- \delta}( M \setminus B_r ) } \notag \\
		& \qquad \quad +  \| \pi \|^2_{ W^{1,p}_{ -1- \delta} ( M \setminus B_r )}\|g^{odd}\|_{ W^{2,p}_{ -1- \delta}( M \setminus B_r ) } \notag \\
		& \qquad \quad + \| g^{ odd } \|_{ W^{2,p}_{ -1- \delta}( M \setminus B_r ) }\| g - \delta \|_{ W^{2,p}_{ - \delta}( M \setminus B_r )} + \| g -\delta \|^2_{ W^{2,p}_{ - \delta} ( M \setminus B_r )} \bigg) r^{-a},
\end{align}
\item[(2) ] 
\begin{align} \label{Sphere}
		&\bigg| \int_{ |x| = r } F(g) \cdot \nu_g \, d\sigma_g - \int_{ |x| = r } F ( \bar{g}_k )\cdot \nu_{ \bar{g}_k} \, d\sigma_{ \bar{g}_k} \bigg| \notag\\
&\le c \bigg( \| 1- u \|_{W_{ -\delta }^{3,p} (M \setminus B_r) }+\| g - \bar{g}_k \|_{W^{3,p}_{-1 -\delta } ( M \setminus B_r )} \bigg) r^b.
\end{align}
\end{itemize}
Then given $\epsilon > 0$, there exists $k_0$ such that  
$$
		\lim_{ r \rightarrow \infty } \int_{ |x| = r } F(g) \cdot \nu_g \,d \sigma_g = \lim_{ r \rightarrow \infty } \int_{ |x| = r} F( \bar{g}_k ) \cdot \nu_{\bar{g}_k } \,d \sigma_{\bar{g}_k } + \epsilon,  \quad \mbox{ for all } k > k_0.
$$

\end{lemma}
\begin{rmk}
If we replace the normal vectors and the volume forms of the integrals in the above assumptions by those with respect to the induced metric in the Euclidean space, the analogous result is the following: assume the following estimates hold
\begin{itemize}
\item[(1) ] 
\begin{align} \label{AnnulusE}
		&\bigg| \int_{ |x| \ge r } {\rm div}_{\delta} F(g) \,d{\rm vol}_0 \bigg| \le c \bigg(  \| \pi^{ even } \|_{ W^{1,p}_{ -2 - \delta } ( M \setminus B_r )} \| \pi \|_{ W^{1,p}_{ -1- \delta}( M \setminus B_r ) } \notag \\
		& \qquad \quad +  \| \pi \|^2_{ W^{1,p}_{ -1- \delta} ( M \setminus B_r )}\|g^{odd}\|_{ W^{2,p}_{ -1- \delta}( M \setminus B_r ) } \notag \\
		& \qquad \quad + \| g^{ odd } \|_{ W^{2,p}_{ -1- \delta}( M \setminus B_r ) }\| g - \delta \|_{ W^{2,p}_{ - \delta}( M \setminus B_r )} + \| g - \delta \|^2_{ W^{2,p}_{ - \delta} ( M \setminus B_r )} \bigg) r^{-a},
\end{align}
\item[(2) ] 
\begin{align} \label{SphereE}
		&\bigg| \int_{ |x| = r } F(g) \cdot \nu_0 \, d\sigma_0 - \int_{ |x| = r } F ( \bar{g}_k )\cdot \nu_0 \, d\sigma_0 \bigg|  \notag\\
		&\le c \bigg( \| 1- u \|_{W_{ -\delta }^{3,p} (M \setminus B_r) } + \| g - \bar{g}_k \|_{W^{3,p}_{-1 -\delta } ( M \setminus B_r )} \bigg) r^b.
\end{align}
\end{itemize}
Then given $\epsilon > 0$, there exists $k_0$ such that
$$
		\lim_{ r \rightarrow \infty } \int_{ |x| = r } F(g) \cdot \nu_0 \,d \sigma_0 = \lim_{ r \rightarrow \infty } \int_{ |x| = r} F( \bar{g}_k ) \cdot \nu_0 \,d \sigma_0 + \epsilon,  \quad \mbox{ for all } k > k_0.
$$
\end{rmk}

%%%%%%%%%%%%%%%%%%%%%%%%%%%%%%%%%%%%%%%%%%%%%%%%%%%%%%%%%%%%%%%%%%%%%%%%%%%%%

\section{Properties of the Intrinsics Definition}
Bartnik \cite[Section $4$]{Bartnik} proved that the ADM mass is a geometric invariant and satisfies natural properties. Along the same lines , we will show that the intrinsic definition of center of mass (\ref{CTM}) is well-defined and has corresponding change of coordinate under the transformation at infinity. We will first show that the intrinsic definition is robust in the sense that we can integrate over a general class of surfaces, and those integrals converge to the same vector. 

%%%%%%%%%%%%%%%%%%%%%%%%%%%%%% Integrate over Generalized surfaces %%%%%%%%%%%%%%%%%%%%%%%%%%%%%%%%%%%%%%%%%%%%%%%

\begin{prop} \label{GeneSur}
Suppose $( M , g , K ) $ is AF-RT (\ref{AF-RT}). Let $\{ D_k \}_{k = 1}^{\infty} \subset M$ be closed sets such that the sets $ S_k = \partial D_k$ are connected two-dimensional $C^1$-surfaces without boundary which satisfy
\begin{align}
&\qquad    r_k = \inf \{ |x| : x\in S_k \} \rightarrow \infty \quad \mbox{as } k \rightarrow \infty, \label{BaC1}&\\
&\qquad    r_k^{-2} {\rm area}(S_k)  \mbox{ is bounded as }  k \rightarrow \infty,  \label{BaC2} &\\
&\qquad    {\rm vol}\{ D_k \setminus D_k^-\} = O( r_k ^{3^-}), \notag & \\
&\qquad \quad   \mbox{where } \,D_k^- = D_k \cap \{- D_k\} \mbox{ and  $3^-$ is a number less than $3$ }. \label{cond}&
\end{align}
Then the center of mass defined by 
\begin{align*}
		C_I^{\alpha} = \frac{1}{ 16\pi m} \lim_{ k \rightarrow \infty} \int_{ S_k } \big( R_{ij} - \frac{1}{2} R_g g_{ij} \big) Y_{(\alpha)}^i \nu_g^j \,d{\sigma}_g , \qquad \alpha = 1, 2, 3
\end{align*}
is independent of the sequence $\{ S_k \}$.
\end{prop}
\begin{Remark}
The first two conditions (\ref{BaC1}), (\ref{BaC2}) on $S_k$ are the conditions considered by Bartnik \cite{Bartnik} to ensure the ADM mass is well-defined. The volume growth condition (\ref{cond}) allows us to consider a general class of surfaces which are roughly symmetric; that is,  the non-symmetric region $D_k \setminus D_k^-$ of $D_k$ has the volume growth slightly less than the volume growth of arbitrary regions in $M$. 
\end{Remark}
\begin{proof}
As in (\ref{Ann}), the divergence theorem gives us
\begin{align*}
		&\int_{ S_k } \big( R_{ij} - \frac{1}{2} R_g g_{ij} \big) Y_{(\alpha)}^i \nu_g^j \,d{\sigma}_g = \int_{ S_1 } \big( R_{ij} - \frac{1}{2} R_g g_{ij} \big) Y_{(\alpha)}^i \nu_g^j \,d{\sigma}_g\\
		 &+ \int_{D_k \setminus D_1} x^{\alpha} R_g + 2 x^{\alpha} (g^{ij} - \delta^{ij} ) \big( R_{ij} - \frac{1}{2} R_g g_{ij} \big) 
		+  \big( R_{ij} + \frac{1}{2} R_g g_{ij} \big)  Y_{(\alpha)}^k  \Gamma_{kj}^i \, d {\rm vol}_g.
\end{align*}
We can decompose the integral over $D_k \setminus D_1 = \{ D_k \setminus D_k^- \} \cup \{D_k^{-} \setminus D_1^+\} \cup \{ D_1^+ \setminus D_1\}$ into three integrals, where $D_1^+ = D_1 \cup \{ - D_1 \}$. The first integral over $\{ D_k \setminus D_k^- \}$ converges because of the volume growth condition (\ref{cond}). The second integral over $\{D_k^{-} \setminus D_1^+\}$ converges because the region is centrically symmetric and the initial data set is AF-RT. After taking limits, the right hand side has a limit independent of the the sequence $\{S_k\}$.

\end{proof}
%%%%%%%%%%%%%%%%%%%%%%%%%%%%%%%%&&&& Change of Coordinate &&&&&&&&&&&&&&&&&&&&&&&&&%%%%%%%%%%%%%%%%%%%%%&&&&&&&&&&&&&

Assume that $\{ x \}$ and $ \{  y \}$ are two AF coordinates on $ ( M \setminus B_{R_0}, g) $. Assume that $F$ is the transition function between these two coordinates and $y = F(x)$. It is shown that the only possible coordinate changes for AF manifolds at infinity are rotation and translation \cite[Corollary $3.2$]{Bartnik}. More precisely, there is a rigid motion of $\mathbb{R}^3, (\mathcal{O}^i_j, a) \in O( 3, \mathbb{R} ) \times \mathbb{R} $ so that 
$$
		\left| F(x) - ( \mathcal{O} x + a ) \right| \in W^{2, \infty}_{0} (\mathbb{R}^3 \setminus B_{R_0}).
$$
Similarly, if we use the same argument in that corollary for AF-RT manifolds, we derive 
$$
		F^{odd}(x) \in W^{2,\infty}_{-1}( \mathbb{R}^3 \setminus B_{R_0}).
$$
Because the center of mass is a quantity depending on the coordinates, we now show that centers of mass in $\{x\}$ and $\{y\}$ coordinates have the corresponding translation and rotation. An interesting phenomenon is that compared to rotation, translation is a more subtle rigid motion. If the translation $a$ is not zero, Density Theorem (Theorem \ref{DenThm2}) is involved in the proof.
\begin{thm} \label{ChangeCoord}
Let $\{ x \}$ and $ \{  y \}$ be two distinct AF-RT (\ref{AF-RT}) coordinates for $ (M, g, K)$ satisfying the change of coordinates as we described above. Assume that $ C_{I, x} $ and $ C_{I,y} $ are the centers of mass defined by the intrinsic definition (\ref{CTM}) in these two coordinates, then 
$$ 
		C_{I, y} = \mathcal{O} C_{I, x} + a.
$$ 
\end{thm}
\begin{proof}
The metric $g$ in the cotangent spaces induced from $\{ x \}$ and $ \{  y \}$ can be written locally as
\begin{align*}
		ds^2 & = g_{kl}(x) dx^k dx^l = \widetilde{g}_{ij}(y) dy^i dy^j \\
		&= \widetilde {g}_{ij} \big( F(x) \big) d( \mathcal{O}^i_{k} x^k + a^i ) d( \mathcal{O}^j_{l} x^l + a^j) + O( |x|^{-1}).
\end{align*}
Expanding the above terms, we get  
$$
		g_{kl}(x) = \widetilde{g}_{ij} \big( F(x) \big) \mathcal{O}^i_k \mathcal{O}^j_l + e_1(x)
$$
where the error term $e_1 \in W^{1,\infty}_{-1}$ and $e_1^{odd} \in W^{1,\infty}_{-2}$. Then a straightforward calculation gives us 
$$
		\frac{ \partial^2 g_{kl} }{ \partial x^m \partial x^n} = \frac{\partial^2 \widetilde{g}_{ij}}{ \partial y^p \partial y^q} \mathcal{O}^i_k \mathcal{O}^j_l \mathcal{O}^p_m \mathcal{O}^q_n + e_2,
$$
and the formulas for Ricci and scalar curvatures, 
\begin{align*}
		& R_{kl}(x) = \widetilde{R}_{ij}\big(  F(x) \big) \mathcal{O}^i_k \mathcal{O}^j_l + e_3(x)\\
		& R(x) = \widetilde{R}\big(  F(x) \big) + e_4(x)
\end{align*}
where $e_q(x) \in O(|x|^{-4})$ and $ e_q^{odd} (x) = O( |x|^{-5}) $ for $q = 2,3,4.$ To calculate $C_{I,y}$, Proposition {\ref{GeneSur}} allows us to integrate over $\left\{| y -  a | = r\right\}$, 
$$
		C^{\alpha}_{I,y} = \frac{1 }{ 16 \pi m}\lim_{ r\rightarrow \infty }\int_{ |y - a| = r} \Big( \widetilde{R}_{ ij }(y) - \frac{1}{2} \widetilde{R}(y) \widetilde{g}_{ij}(y) \Big) \Big( |y|^2 \delta^{\alpha i } - 2 y^{\alpha} y^i\Big) \frac{ y^j - a^j}{ | y - a | } \, d\sigma_0.
$$
We would like to replace $ y = F(x) $ in the above identity. First we have
\begin{eqnarray*}
		 | y |^2 \delta^{\alpha i } - 2 y^{\alpha} y^i &=& \big( |x|^2 - 2 \mathcal{O}^{\alpha}_{\beta} x^{ \beta} \mathcal{O}^i_k x^k \big)\delta^{\alpha i} +  2 \mathcal{O}x \cdot a  \delta^{\alpha i } \\
		&& - 2\big( a^{\alpha} \mathcal{O}^i_k x^k + \mathcal{O}^{\alpha}_{\beta} x^{ \beta}  a^i \big) + O(1 ),\\
		\frac{ y^j - \alpha^j }{ | y - \alpha | } &=& \frac{ \mathcal{O}^j_l x^l}{ | x | }  + e_5,
\end{eqnarray*}
where $e_5 \in W^{2,\infty}_{-1}$ and $e_5^{odd} \in W^{2,\infty}_{-2}$. Using the facts that $\widetilde{R}_{ij}(F), \widetilde{R}(F)$ and $g(F)$ are asymptotically even 
\begin{align*}
		&\int_{ |y - a| = r} \Big( \widetilde{R}_{ ij }(y) - \frac{1}{2} \widetilde{R}(y) \widetilde{g}_{ij}(y) \Big) \Big( |y|^2 \delta^{\alpha i } - 2 y^{\alpha} y^i\Big) \frac{ y^j - a^j}{ | y - a | } \,d\sigma_0 \\
		=&\int_{ |x| = r} \Big( \widetilde{R}_{ ij }\big( F(x) \big) - \frac{1}{2} \widetilde{R} \big( F( x ) \big) \widetilde{g}_{ij}\big( F(x) \big) \Big) \Big( | x |^2 \delta^{\alpha i } 
		- 2 \mathcal{O}^{\alpha}_{\beta} x^{ \beta} \mathcal{O}^i_k x^k\Big) \frac{ \mathcal{O}^j_l x^l}{ | x | } \, d\sigma_0\\
		& +  \int_{ |x| = r} \Big( \widetilde{R}_{ ij }\big( F(x) \big)  - \frac{1}{2} \widetilde{R}\big( F(x) \big) \widetilde{g}_{ij}\big( F(x) \big) \Big) \Big(   - 2 a^{\alpha} \mathcal{O}^i_k x^k \frac{ \mathcal{O}^j_l x^l}{ | x | } \Big) \, d\sigma_0 \\
		&+  \int_{ |x| = r} \Big( \widetilde{R}_{ ij }\big( F(x) \big) - \frac{1}{2} \widetilde{R}\big( F(x) \big)  \widetilde{g}_{ij}\big( F(x) \big) \Big)
		\Big( 2 \mathcal{O}x \cdot a  \delta^{\alpha i }- 2 \mathcal{O}^{\alpha}_{\beta} x^{ \beta}  a^i )\Big) \frac{ \mathcal{O}^j_l x^l}{ | x | }\, d\sigma_0\\
		& + O( r^{-1}).
\end{align*}

We claim that the first integral $I_1$ converges to $ 16 \pi m \mathcal{O}^{\alpha}_{\beta} C^{\beta}_{I,x}$, the second integral $I_2$ converges to $16 \pi m a^{\alpha}$, and the third integral $I_3$ converges to $0$. 
\begin{align*}
		I_1 & = \int_{ |x| = r} \Big( \widetilde{R}_{ ij }\big( F(x) \big) - \frac{1}{2} \widetilde{R}\big( F(x) \big) \widetilde{g}_{ij}\big( F(x) \big)  \Big) \mathcal{O}^i_k\mathcal{O}^j_l \Big( | x |^2 \mathcal{O}^{\alpha}_k 
		- 2 \mathcal{O}^{\alpha}_{\beta} x^{ \beta} x^k\Big) \frac{  x^l}{ | x | } \,d\sigma_0\\
		& = \mathcal{O}^{\alpha}_{\beta}\int_{ |x| = r}  \Big( R_{ kl }( x ) - \frac{1}{2} R_g(  x ) g_{kl}( x) \Big) \Big( | x |^2 \delta^{\beta k}- 2 x^{ \beta} x^k\Big) \frac{  x^l}{ | x | } \,d\sigma_0 + O( r^{-1})\\
		& =16 \pi m  \mathcal{O}^{\alpha}_{\beta} C^{\beta}_{I,x} + O( r^{-1}). 
\end{align*}
By using the formula for ADM mass (\ref{ADMmass}), 
\begin{align*}
		I_2 &=a^{\alpha} \int_{ |x| = r} \Big( R_{ kl }( x ) - \frac{1}{2} R_g(  x ) g_{kl}( x)  \Big) \Big(   - 2  x^k \frac{  x^l}{ | x | } \Big) \,d\sigma_0+ O( r^{-1}) \\
		&= 16 \pi m a^{\alpha} + O ( r^{-1}).	
\end{align*}
To see that $I_3$ converges to $0$, we simplify the expression by letting $ b = \mathcal{O}^Ta$, and then
$$
		I_3  = \mathcal{O}^{\alpha}_{\beta} \int_{ |x| = r}  \Big( R_{ kl }( x ) - \frac{1}{2} R_g(  x ) g_{kl}( x) \Big) \Big( 2 x \cdot b \delta^{\beta k} - 2 x^{\beta} b^k \Big) \frac{ x^l }{ |x| } \,d\sigma_0.
$$
By the fact that $ \big( 2 x \cdot b \delta^{\beta k} - 2 x^{\beta} b^k \big) \cdot \frac{x^k } { |x| } = 0 $, the scalar curvature term vanishes after taking a limit, so we only need to prove
$$
		\mathcal{O}^{\alpha}_{\beta} \int_{ |x| = r}  R_{ kl }( x )\Big( 2 x \cdot b \delta^{\beta k} - 2 x^{\beta} b^k \Big) \frac{ x^l }{ |x| } \,d\sigma_0 = O( r^{-1}).
$$
It is sufficient to prove
$$
		\int_{ |x| = r}  R_{ kl }( x )\Big( 2 x \cdot b \delta^{\beta k} - 2 x^{\beta} b^k \Big) \nu_g^l \, d\sigma_g = O(r^{-1}).
$$
A straightforward calculation shows this identity holds if the metric is conformally flat outside a compact set, so we would like to approximate $g$ by $\bar{g}$ which have harmonic asymptotics, and then apply Lemma \ref{Conv}. We need to check if conditions (\ref{Annulus}) and (\ref{Sphere}) hold. Using the divergence theorem, we have
\begin{align*}
		&\int_{ |x| \ge r }  {\rm div}_g \left[ R_{ kl }( x )\big( 2 x \cdot b \delta^{\beta k} - 2 x^{\beta} b^k \big) \right] \, d {\rm vol}_g \\
		 = &\int_{ |x| \ge r } ( {\rm div}_g R_{kl} )\big( 2 x \cdot b \delta^{\beta k} - 2 x^{\beta} b^k \big) + \underbrace{R_{kl} \big(  2 b^l \delta^{\beta k} - 2 \delta^{\beta l } b^k \big)}_{\mbox{ sum to }= 0} \, d{\rm  vol}_g\\
		 &+\int_{ |x| \ge r }   R_{kl } \big( 2 x \cdot b \delta^{\beta k} - 2 x^{\beta} b^k \big)  g^{ln}\Gamma_{mn}^m \, d {\rm vol}_g.
\end{align*}
Using the second Bianchi identity to the first term on the right hand side, symbolically, the above integral is bounded by
\begin{align*}
		&\left| \int_{ |x| \ge r } {\rm div}_g \left[ R_{ kl }( x )\big( 2 x \cdot b \delta^{\beta k} - 2 x^{\beta} b^k \big) \right] \, d {\rm vol}_g  \right| \\
		&\le c\Big( \| \pi\|_{W^{1, p}_{-1-\delta} ( M \setminus B_{r} )} \| \pi^{even} \|_{W^{1,p}_{ -2 - \delta } ( M \setminus B_{r} ) } + \| g\|^2_{W^{2,p}_{-\delta}(M \setminus B_r)} \Big) r^{1 - 2 \delta}.
\end{align*}
Hence condition (\ref{Annulus}) holds. Condition (\ref{Sphere}) holds because the Sobolev inequality implies
\begin{align*}
		&\int_{ |x| = r}  R_{ kl }( x )\left( 2 x \cdot b \delta^{\beta k} - 2 x^{\beta} b^k \right) \nu_g^l \, d\sigma_g - \int_{ |x| = r}  \bar{R}_{ kl }( x )\left( 2 x \cdot b \delta^{\beta k} - 2 x^{\beta} b^k \right) \nu_{\bar{g}}^l \, d\sigma_{\bar{g}} \\
		& \le c \left( \| 1 - u \|_{W^{3,p}_{-\delta} (M \setminus B_r)} + \| g - \bar{g} \|_{ W^{3,p}_{-\delta}(M \setminus B_r) } \right) r^{1-\delta}.
\end{align*}
By Lemma \ref{Conv}, we conclude 
\begin{align*}
		\lim_{ r \rightarrow \infty } \int_{ |x| = r}  R_{ kl }( x )\Big( 2 x \cdot b \delta^{\beta k} - 2 x^{\beta} b^k \Big) \nu_g^l \, d\sigma_g \\
		= \lim_{ r \rightarrow \infty } \int_{ |x| = r}  \bar{R}_{ kl }( x )\Big( 2 x \cdot b \delta^{\beta k} - 2 x^{\beta} b^k \Big) \nu_{\bar{g}}^l \,d\sigma_{\bar{g}} + \epsilon = \epsilon.
\end{align*}
Since $\epsilon$ can be chosen arbitrarily small, $I_3$ converges to $0$. Therefore, 
$$
		C^{\alpha}_{I, y} = \frac{1 }{ 16 \pi m} \lim_{r \rightarrow \infty} (I_1 + I_2 + I_3) = \mathcal{O}^{\alpha}_{ \beta} C^{\beta}_{I,x} + a^{\alpha}.
$$

\end{proof}

%%%%%%%%%%%%%%%%%%%%%%%%%%%%%%%%%Corvino_Schoen Center of Mass%%%%%%%%%%%%%%%%%%%%%%%%%%%%%%%%%%%%%%
 
\section{ The Corvino--Schoen Center of Mass }
Corvino and Schoen \cite{CS} defined the center of mass (\ref{CS}) for AF-RT manifolds. In this section, we will show that the intrinsic definition we propose is actually equal to the Corvino--Schoen definition (Theorem {\ref{CorvinoSchoen}}). In other words, the intrinsic definition (\ref{CTM}) is a coordinate-free expression of the Corvino--Schoen definition. 

Assume $\bar{g}$ is the approximating solution in Theorem {\ref{DenThm2}} (Density Theorem). A straightforward calculation shows $C_I (\bar{g}) = C_{CS}( \bar{g})$. As proven in Theorem {\ref{DenThm2}}, $C_I(g)$ can be approximated by $C_I ( \bar{g} )$. If we prove that $C_{CS} (g) $ can also be approximated by $C_{CS} (\bar{g})$, then Theorem {\ref{CorvinoSchoen}} follows.

\begin{proof}[Proof of Theorem 1]
First notice that the Corvino--Schoen definition is equal to the following expression where the normal vector and the volume form are with respect to the induced metric in the Euclidean space
\begin{align*}
	C_{CS}^{\alpha} =\frac{1}{ 16 \pi m} \left[ \lim_{ r \rightarrow \infty} \int_{ |x| = r}  x^{\alpha} \sum_{i,j} ( g_{ij,i} - g_{ii,j} ) \nu_0^j\,d\sigma_0\right.\\
	\left. - \int_{ |x| = r}  \sum_i ( h_{i\alpha} \nu_0^i - h_{ii} \nu_0^{\alpha} ) \,d\sigma_0 \right]. 
\end{align*}
A direct calculation shows
\begin{align*}
		 &\left[\int_{ |x| = r_2} x^{\alpha} \sum_{i,j} ( g_{ij,i} - g_{ii,j} ) \nu_0^j  \, d\sigma_0 -\int_{ |x| = r_2} \sum_i (  h_{i\alpha} \nu_0^i - h_{ii} \nu_0^{\alpha}) \, d\sigma_0\right] \\
		 &\quad - \left[ \int_{ |x| = r_1}  x^{\alpha} \sum_{i,j} ( g_{ij,i} - g_{ii,j} ) \nu_0^j\, d\sigma_0 - \int_{ |x| = r_1}  \sum_i ( h_{i\alpha} \nu^i - h_{ii} \nu_0^{\alpha} ) \, d\sigma_0 \right] \\
		 &= \int_{ r_1 \le |x| \le r_2  } x^{\alpha} \sum_{i,j} (g_{ij,ij} - g_{ii,jj}) \,d {\rm vol}_0
\end{align*}
It is easy to see that the conditions (\ref{AnnulusE}) and (\ref{SphereE}) in Lemma \ref{Conv} (and the remark afterward) hold because $\sum_{i,j} (g_{ij,ij} - g_{ii,jj})$ is the leading order term of $R_g$, and by the constraint equation (\ref{PiCE}),
\begin{align*}
		&\left| \int_{  |x| \ge r } x^{\alpha} \sum_{i,j}(g_{ij,ij} - g_{ii,jj}) \,d{\rm vol}_0 \right| \le c \left( \| \pi \|_{W^{1,p}_{-1 - \delta} ( M \setminus B_r)} \| \pi^{even} \|_{W^{1,p}_{-2 - \delta}( M \setminus B_r)} \right. \\
		& \qquad \left. + \| g^{odd} \|_{W^{2,p}_{-1-\delta} (M \setminus B_r)} \| g - \delta \|_{W^{2,p}_{-\delta}(M \setminus B_r)}\right) r^{ 1 - 2 \delta}
\end{align*}
and
\begin{align*}
		 &\left| \bigg( \int_{ |x| = r}  x^{\alpha} \sum_{i,j} ( g_{ij,i} - g_{ii,j} ) \nu_0^j\,d\sigma_0 - \int_{ |x| = r}  \sum_i ( h_{i\alpha} \nu_0^i - h_{ii} \nu_0^{\alpha}) \,d\sigma_0 \bigg) \right. \\
		&-  \left.\bigg(\int_{ |x| = r}  x^{\alpha} \sum_{i,j} ( \bar{g}_{ij,i} - \bar{g}_{ii,j} ) \nu_0^j\,d\sigma_0 - \int_{ |x| = r}  \sum_i ( \bar{h}_{i\alpha} \nu_0^i - \bar{h}_{ii} \nu_0^{\alpha}) \,d\sigma_0\bigg) \right| \\
		&\le c \| g - \bar{g} \|_{ W^{3, p }_{ - \delta} ( M \setminus B_r) } r^{2 - \delta}.  
\end{align*}
Then we can apply Lemma \ref{Conv} and have
\begin{align*}
		C^{\alpha}_{CS} &=\frac{1}{16 \pi m}  \lim_{ r \rightarrow \infty} \left[ \int_{ |x| = r}  x^{\alpha} \sum_{i,j} ( g_{ij,i} - g_{ii,j} ) \nu_0^j\,d\sigma_0 - \int_{ |x| = r}  \sum_i ( h_{i\alpha} \nu_0^i - h_{ii} \nu_0^{\alpha}) \,d\sigma_0\right] \\
		&=\frac{1}{16 \pi m}  \lim_{ r \rightarrow \infty} \left[ \int_{ |x| = r}  x^{\alpha} \sum_{i,j} ( \bar{g}_{ij,i} - \bar{g}_{ii,j} ) \nu_0^j\,d\sigma_0 - \int_{ |x| = r}  \sum_i ( \bar{h}_{i\alpha} \nu_0^i - \bar{h}_{ii} \nu_0^{\alpha}) \,d\sigma_0\right] + \epsilon \\
		& = C_{CS}^{\alpha} ( \bar{g})+ \epsilon.
\end{align*}
\end{proof}

%%%%%%%%%%%%%%%%%%%%%%%%%%%%%%%%%%%%%%%%%%%%%%%%%%%%%%%%%%%%%%%%%%%%%%%%%%%%%%%
\section{The Huisken--Yau Center of Mass}
In the case that $(M, g, K)$ is SAF (\ref{SAF}), Huisken--Yau \cite{HY}  and Ye \cite{Y} proved the existence and uniqueness of the constant mean curvature foliation in the exterior region of $M$, if $m > 0$. Huisken and Yau used the volume preserving mean curvature flow to evolve each Euclidean sphere centered at the origin with a large radius and they showed that the Euclidean sphere converges to a surface with constant mean curvature. Furthermore, they proved those surfaces with constant mean curvature are approximately round, and their approximate centers converge as per equation \eqref{HY}:
\begin{align}
		C^{\alpha}_{HY}=\lim_{r \rightarrow \infty} \frac{ \int_{M_r } z^{\alpha } \,d\sigma_0 }{ \int_{ M_r } \, d\sigma_0 }, \quad \alpha = 1,2,3 \tag{1.7}
\end{align}
where $\{ M_r \}$ are leaves of the foliation and $z$ is the position vector. They define this to be the center of mass for $M$. Instead of using volume preserving mean curvature flow, Ye used the implicit function theorem to find a surface with constant mean curvature which is perturbed in the normal direction away from the Euclidean sphere $S_R(p)$ for some suitable center $p$ and for the radius $R$ large. Although it is not emphasized in his paper, $p$ can represent the center for each constant mean curvature surface because this surface is roughly round. Furthermore, we will show that $p $ converges to $C_{CS}$ and use this fact to prove $C_{HY}= C_{CS} $. Then $C_I = C_{HY}$ follows because the Corvino--Schoen center of mass $C_{CS}$ is equal to the intrinsic definition $C_I$ by Theorem \ref{CorvinoSchoen}. Before we prove Theorem \ref{HuiskenYau}, we will need a technical lemma (Lemma \ref{Tech}) which suggests that the center of the surface $p$ converges to $C_{CS}$.

Let $\nu$ be the outward unit normal vector field on $S_R (p) $ with respect to the metric $g$. As $R$ varies, $\nu$ is well-defined in a tabular neighborhood of $S_R(p)$. Therefore, the mean curvature at $y \in S_R(p)$ is ${\rm div}_{S_R(p)} \nu = \textup{div}_g \nu$, the divergence operator of the ambient manifold $M$. Recall that $\displaystyle g = \left( 1 + \frac{2m}{|y|} \right) \delta_{ij} + p_{ij}$. At $y$, 
%We parallel transport $\nu$ along the normal direction so that it is well-defined in a tubular neighborhood of $S_R(p)$. Hence the mean curvature on $y \in S_R ( p ) $ is ${\rm div}_{S_R(p)} \nu =  {\rm div}_{g} \nu $, the divergence operator of the ambient manifold $M$. 
\begin{align*}
		\nu &= \frac{ \nabla | y - p |  }{ \big| \nabla | y - p | \big|_g} \\
		&=  \bigg( 1 - \frac{ m }{ |y| } + \frac{ 3 m^2 }{  2 |y|^2} + \frac{1}{2} p_{qr} \frac{ (y^q - p^q )( y^r - p^r ) }{ | y - p |^2 }\bigg)\frac{ y^l - p^l }{ | y - p| } \frac{ \partial }{ \partial y^l} \\
		&\quad - p_{kl} \frac{ y^k - p^k }{ | y - p| } \frac{ \partial }{ \partial y^l} + O( R^{-3} ).
\end{align*}
A straightforward calculation gives us the mean curvature on $S_R(p)$ is equal to
\begin{eqnarray}
		  {\rm div}_{g} \nu &=&\frac{2}{ | y - p|} - \frac{ 4m }{ | y - p |^2  } + \frac{ 6 m (y-p) \cdot p  }{  | y -  p|^4} + \frac{ 9 m^2 }{ | y - p |^3 } \nonumber \\ 
  &&+ \frac{1}{2} p_{ij,k}(y) \frac{ ( y^i - p^i )(y^j - p^j) (y^k - p^k)  }{ | y -  p|^3 } + 2 p_{ij}(y) \frac{ ( y^i - p^i )(y^j - p^j ) }{ | y - p |^3 } \nonumber \\ 
  &&- p_{ij,i}(y) \frac{ y^j - p^j }{ | y - p | } - \frac{ p_{ii}(y)}{ | y -  p| } + \frac{1}{2} p_{ii,j}(y) \frac{ y^j- p^j }{ | y-  p | } + E_0, \label{MC}
\end{eqnarray}
where $ | E_0 | \le \frac{c}{R^4}\big(1 + |p| \big)$ for some constant $c$ depending only on the metric $g$.

\begin{lemma} \label{Tech}
For $R$ large, 
\begin{eqnarray} \label{pterm}
		&&\int_{ | y - p| = R} ( y^{ \alpha } - p^{ \alpha } ) \bigg(  \frac{1}{2} p_{ij,k} \frac{ ( y^i - p^i )(y^j - p^j) (y^k - p^k)  }{ | y -  p|^3 } \bigg) \, d\sigma_0 \nonumber \\
		&&+ \int_{ | y - p| = R} ( y^{ \alpha } - p^{ \alpha } ) \bigg(  2 p_{ij} \frac{ ( y^i - p^i )(y^j - p^j ) }{ | y - p |^3 } \bigg) \, d \sigma_0 \nonumber\\
		&&+ \int_{ | y - p| = R} ( y^{ \alpha } - p^{ \alpha } ) \bigg( - p_{ij,i} \frac{ y^j - p^j }{ | y - p | } - \frac{ p_{ii}}{ | y -  p| } + \frac{1}{2} p_{ii,j} \frac{ y^j- p^j }{ | y-  p | } \bigg) \, d \sigma_0 \nonumber\\
 		&= & - 8 m \pi C_{CS}^{\alpha} + O(R^{-1}), \qquad \qquad \alpha= 1,2,3. 
\end{eqnarray}
\end{lemma}
\begin{proof}[Proof of Lemma \ref{Tech}]
We denote the first integral by, for  $\alpha = 1, 2, 3$,
$$
				\mathcal{I}^{\alpha} (R)  =  \int_{ | y - p| = R} ( y^{ \alpha } - p^{ \alpha } ) \bigg(  \frac{1}{2} p_{ij,k} \frac{ ( y^i - p^i )(y^j - p^j) (y^k - p^k)  }{ | y -  p|^3 } \bigg) \, d \sigma_0.
$$
In the proof, we will rewrite $\mathcal{I}^{\alpha} (R) $, and then some cancellation allows us to rearrange the left hand side of (\ref{pterm}) so that it has an expression corresponding to $C_{CS}^{\alpha}$. 

Since the coordinate is only defined outside a compact set of $M$, we can use the divergence theorem only  in the annular region $A = \{ R \le | y - p | \le R_1\}$,
\begin{eqnarray*}
		\mathcal{I}^{\alpha}(R_1) &=& \mathcal{I}^{\alpha} (R) + \frac{1}{2} \int_A  \bigg( p_{ij,k} \frac{ ( y^j - p^j ) ( y^k - p^k ) ( y^{\alpha} - p^{\alpha} ) }{ | y - p |^2 }\bigg)_{,i} \, d {\rm vol}_0\\
		&=& \mathcal{I}^{\alpha} (R) + \frac{1}{2} \int_A \bigg( p_{ij,i} \frac{ ( y^j - p^j ) ( y^k - p^k ) ( y^{\alpha} - p^{\alpha} ) }{ | y - p |^2 } \bigg)_{,k} \, d {\rm vol}_0 \nonumber \\
		&& -  \frac{1}{2} \int_A p_{ij,i} \bigg( \frac{ ( y^j - p^j ) ( y^k - p^k ) ( y^{\alpha} - p^{\alpha} ) }{ | y - p |^2 } \bigg) _{,k} \, d {\rm vol}_0 \nonumber\\
		&&  + \frac{1}{2} \int_A p_{ij,k} \bigg( \frac{ ( y^j - p^j ) ( y^k - p^k ) ( y^{\alpha} - p^{\alpha} ) }{ | y - p |^2 } \bigg) _{,i} \, d {\rm vol}_0. 
\end{eqnarray*}
Using the divergence theorem and simplifying the expression, we get an identity containing purely boundary terms
$$
		\mathcal{I}^{\alpha}(R_1) = \mathcal{I}^{\alpha} (R) + \mathcal{J}^{\alpha} (R_1) - \mathcal{J}^{\alpha}(R),  \quad \mbox{ for all } R_1 > R
$$
where 
\begin{eqnarray*}
		\mathcal{J}^{\alpha} (R) &=& \frac{1}{2} \int_{ | y - p | = R} ( y^{\alpha} - p^{\alpha} ) p_{ij,i} \frac{ y^j - p^j }{ | y - p| } \,d \sigma_0 \\
		&&- 2 \int_{ | y - p | = R} ( y^{\alpha} - p^{\alpha} ) p_{ij}\frac{ ( y^i - p^i)( y^j - p^j )  }{ | y - p |^3} \,d \sigma_0\\
		&&+ \frac{1}{2} \int_{ | y - p | = R} p_{ii} \frac{ y^{ \alpha } - p^{ \alpha }  }{ | y - p |} d \sigma_0 + \frac{1}{2} \int_{ | y - p |= R} p_{i \alpha } \frac{ y^i - p^i }{ | y - p |} \,d \sigma_0. 
\end{eqnarray*}
To prove that $\mathcal{I}^{\alpha}(R)= \mathcal{J}^{\alpha}(R)$, we would like to apply Lemma \ref{Conv}. It is easy to check that the conditions (\ref{AnnulusE}) and (\ref{SphereE}) hold, so we get, for $\alpha = 1,2,3$, 
\begin{eqnarray*}
		\mathcal{ I }^{\alpha} (R) - \mathcal{J}^{\alpha} (R) = \lim_{R_1 \rightarrow \infty} (\mathcal{I}^{\alpha} (R_1) -	 \mathcal{J}^{\alpha} (R_1 ) )= \lim_{R_1 \rightarrow \infty} (\mathcal{I}^{\alpha}_{\bar{p}} (R_1) - \mathcal{J}^{\alpha}_{\bar{p} } (R_1))+ \epsilon,
\end{eqnarray*}
where $\mathcal{I}^{\alpha}_{\bar{p} } $ and $ \mathcal{J}^{\alpha}_{\bar{p} }$ denote the integrals that we obtain by replacing $p_{ij}$ by $\bar{p}_{ij}$ in $\mathcal{ I } $ and  $\mathcal{J}^{\alpha} $, where $\bar{p}_{ij}$ are $O(|y|^{-2})$-terms in the approximating solutions $\bar{g}_{ij}$ in Theorem \ref{DenThm2} (Density Theorem). More precisely, $\bar{g}_{ij}$ has this expansion
$$
		\bar{g}_{ij} = \Big( 1 + \frac{a }{ | y |}\Big) \delta_{ij} + O(|y|^{-2}), 
$$
and $\bar{p}_{ij}$ is defined by
$$
		\bar{p}_{ij} = \bar{g}_{ij} - \Big( 1 + \frac{a }{ | y |}\Big) \delta_{ij}.
$$
Because $\bar{p}_{ij}^{odd} = \frac{ c \cdot y}{ |y|^3} + O( |y|^{-3})$ as in (\ref{Asym2}), it is easy to see that
$$
		\lim_{R_1 \rightarrow \infty} (\mathcal{I}^{\alpha}_{\bar{p} } (R_1) - \mathcal{J}^{\alpha}_{\bar{p} } (R_1) )= 0.
$$
Therefore, we conclude $ \mathcal{ I }^{\alpha} (R) = \mathcal{J}^{\alpha} (R) $. We then replace $\mathcal{ I }^{\alpha} (R) $ by $\mathcal{J}^{\alpha} (R) $ in the identity (\ref{pterm}) and derive that the left hand side is equal to
\begin{eqnarray*}
		&&- \frac{1}{2} \left(  \int_{ | y - p| = R}  (y^{\alpha} - p^{\alpha} ) ( p_{ij,i} - p_{ii,j} )\frac{ y^j - p^j }{ | y - p | } \, d\sigma_0 \right.\\
		&&\qquad  \qquad \qquad - \left. \int_{ | y - p| = R}  p_{i\alpha} \frac{ y^i - p^i }{ | x - p |}- p_{ii} \frac{ y^{\alpha} - p^{ \alpha} }{ | x - p |} \,d\sigma_0 \right).
\end{eqnarray*}
We rewrite the above integrals into the summation of the center of mass (\ref{CS}) and the remainder terms. Then explicit calculations yield 
%Because the explicit calculations using $g_{ij}-p_{ij} = \left( 1 + \frac{2m}{|x|}\right) \delta_{ij}$ give us the remainder terms has lower order, we derive
\begin{align*}
		  &- \frac{1}{2} \left[  \int_{ | y - p| = R}  (y^{\alpha} - p^{\alpha} ) ( g_{ij,i} - g_{ii,j} )\frac{ y^j - p^j }{ | y - p | } \,d\sigma_0\right.\\
		  & \qquad \qquad - \left.\int_{ | y - p| = R}  g_{i\alpha} \frac{ y^i - p^i }{ | x - p |}- g_{ii} \frac{ y^{\alpha} - p^{ \alpha} }{ | x - p |} \,d\sigma_0 \right]\\
		 & + \frac{1}{2} \left[  \int_{ | y - p| = R}  (y^{\alpha} - p^{\alpha} ) \left(( g_{ij,i}-p_{ij,i} )- ( g_{ii,j} -p_{ii,j} )\right)\frac{ y^j - p^j }{ | y - p | } \,d\sigma_0 \right] \\
		 & - \frac{1}{2} \left[   \int_{ | y - p| = R}  ( g_{i \alpha } -p_{i\alpha} )\frac{ y^i - p^i }{ | x - p |}- ( g_{ii} - p_{ii}) \frac{ y^{\alpha} - p^{ \alpha} }{ | x - p |} \,d\sigma_0 \right]\\
		 &= - 8 \pi m C_{CS}^{\alpha} + O( R^{-1}).
\end{align*}
\end{proof}

\begin{proof}[Proof of Theorem \ref{HuiskenYau}]
Let $F_{p,R} : S_1 (0) \rightarrow M $ be an embedding defined by $ y = F_{p,R} (x) = R x + p$, that is, $F_{p,R} (S_1(0) ) = S_R (p)$, the Euclidean sphere centered at $p$ with the radius $R$ in $M$. We consider the perturbation along the normal direction on $S_R(p)$ defined by $\Sigma = \{ y + \lambda \phi\nu  : y \in S_R(p) \}$ for a parameter $\lambda > 0$, and for $\phi \in C^{ 2, \alpha }( S_R(p) )$ with $ \| \phi\|_{ C^{2, \alpha} }\le 1$. We denote the mean curvature on $\Sigma$ by $ H(p, R, \lambda \phi) $. Using this notation, $H(p, R, 0) = {\rm div}_g \nu $, the mean curvature of $S_R (p )$. By Taylor's theorem for mappings between two Banach Spaces, %we have the following expansion in the $\phi$-component at $0$, 
\begin{eqnarray*}
	H(p,R, \lambda \phi ) &=& H( p, R, 0) + d H(p, R, 0) \lambda \phi \nonumber\\ 
	&&+ \int_0^1 (1-s) \Big(d^2 H \big (p, R, s(  \lambda \phi) \big) (\lambda \phi, \lambda \phi) \Big) ds
\end{eqnarray*}
where $ d H$ and $ d^2 H$ are the first and second Fr{\'e}chet derivatives in the $\phi$-component. In our case, $ d H(p, R, 0)  $ is the linearized mean curvature operator on $S_R(p) $, i.e. 
$$
		d H(p, R, 0) = \Delta_{S_R(p)} + |A|_g^2 + Ric(\nu, \nu),
$$
where $\Delta_{S_R(p)}$ is the Laplacian operator on $S_R (p) $ with respect to the induced metric from $g$, $A$ is the second fundamental form on $S_R(p)$ and $Ric(\cdot, \cdot) $ is the Ricci curvature of $M$. Then
\begin{eqnarray}
	H(p,R, \lambda \phi ) &=& H( p, R, 0) + \Delta_{S_R(p)} \lambda \phi  + \left( |A|_g^2 + Ric(\nu, \nu)\right)  \lambda \phi  \nonumber\\ 
	&&+ \int_0^1 (1-s) \Big(d^2 H \big (p, R, s(  \lambda \phi) \big) \lambda^2 \phi \phi \Big) ds. \label{GenMC}
\end{eqnarray}

In \cite{HY}, the estimates on the eigenvalues $\lambda_1$ and $\lambda_2$ of $A$ are derived and
$$
		| A|_g^2 = \lambda_1^2 + \lambda_2^2= \frac{2}{R^2} + O( R^{-3}), \quad Ric( \nu, \nu ) = O( R^{-3}).
$$
For the second Fr{\'e}chet derivative in the Taylor expansion, we have 
$$
		d^2 H \big(p, R, s (\lambda \phi) \big)  \lambda^2 \phi \phi = \frac{ \partial^2  }{ \partial t^2} H\big(p, R, t (\lambda \phi) \big) \Big|_{ t= s}.
$$
The right hand side is the second derivative of the mean curvature of the surface $\{y + s (\lambda \phi ) \nu : y\in S_R(p)\}$. For $R$ large, the unit outward normal vector field on $\{y + s ( \lambda \phi ) \nu : y\in S_R(p)\}$ is close to $\nu$, and a straightforward calculation gives us 
$$
		\bigg|\frac{ \partial^2  }{ \partial t^2} H \big(p, R, t ( \lambda \phi) \big ) \bigg| \le c \lambda^2 (| R_{ijkl} | |A| | \phi|^2 + |A| | \phi | | D^2 \phi| + |A|^3 |\phi|^2 )
		%		\bigg( \frac{ \lambda |\phi|}{ R^3 } + \frac{ \lambda^2 | D \phi | | \phi | }{ R^2 } + \frac{ \lambda^2 | D \phi |^2 }{ R } + \lambda^3 | \phi| |D\phi| |D^2 \phi| \bigg), 
$$
where the constant $c$ is independent of $p, R, \phi $. 
%To manage those terms in the Taylor expansion of the mean curvature for $\Sigma$,
%\begin{eqnarray}
%	H(p,R, \lambda \phi ) &=& H( p, R, 0) + \Delta_{S_R(p)} \lambda \phi  + \left( |A|_g^2 + Ric(\nu, \nu)\right)  \lambda \phi  \nonumber\\ 
%	&&+ \int_0^1 (1-s) \Big(d^2 H \big (p, R, s(  \lambda \phi) \big) \lambda^2 \phi \phi \Big) ds. \label{GenMC}
%\end{eqnarray}
Let $G$ and $E_1$ be defined as follows, where $G$ is the lower order terms of the mean curvature of $S_R(p)$ from (\ref{MC}), 
\begin{eqnarray*}
	G (y)&=& \frac{1}{2} p_{ij,k} \frac{ ( y^i - p^i )(y^j - p^j) (y^k - p^k)  }{ | y -  p|^3 } + 2 p_{ij} \frac{ ( y^i - p^i )(y^j - p^j ) }{ | y - p |^3 } \nonumber \\ 
  &&- p_{ij,i} \frac{ y^j - p^j }{ | y - p | } - \frac{ p_{ii}}{ | y -  p| } + \frac{1}{2} p_{ii,j} \frac{ y^j- p^j }{ | y-  p | } ,
\end{eqnarray*}
and
\begin{align*}
		E_1(y) = &E_0 + \Big(| A|_g^2 - \frac{2}{R^2} \Big) (\lambda \phi) + Ric ( \nu, \nu ) (\lambda \phi) \\
		&+  \int_0^1 (1-s) \Big(d^2 H\big(p, R, s \lambda \phi) \big) (\lambda\phi, \lambda \phi) \Big) ds.
\end{align*} 
$G$ will give us $C_{CS}$ as indicated in Lemma \ref{Tech} and $E_1(y) $ is an error term bounded as follows for some constant $c$ independent of $p, R, \phi$.
$$
		| E_1 | \le \frac{c}{R^4} (1 + |p|) + \frac{c}{R^3} \bigg( \lambda  | \phi | +  \lambda^2 | \phi | + R^2\lambda^2 |\phi| | D^2 \phi | \bigg).
$$
From identities (\ref{MC}) and (\ref{GenMC}),  
$$
		H(p, R, \lambda \phi ) = \frac{2}{R} - \frac{4 m}{ R^2 } + \frac{ 6m ( y - p ) \cdot p }{ R^4 } + \frac{ 9m^2}{ R^3} + G(y) + \lambda \Delta_{S_R(p)} \phi + \frac{2}{R^2} \lambda \phi + E_1(y).
$$
To find the surface with constant mean curvature, we need to find  $p, R, \phi$ so that 
$$
H(p, R, \lambda \phi ) = \frac{2}{R} - \frac{4 m}{ R^2 }. 
$$
It is equivalent to solving
\begin{eqnarray}
		0 = \frac{ 6 m (y-p) \cdot p  }{  R^4} + \frac{ 9 m^2 }{ R^3 } + G(y) + \lambda \Delta_{S_R(p)} \phi + \frac{2}{R^2} \lambda  \phi + E_1(y) \label{Zero}
\end{eqnarray}

We pull back the equation (\ref{Zero}) via the map $F_{p,R}$, and we get the following equation on $S_1(0)$,
\begin{eqnarray}
		0 &=& \frac{ 6 m x \cdot p  }{  R^3} + \frac{ 9 m^2 }{ R^3 } + G\circ F_{p,R}(x) +\lambda \Delta_{S_1(0)} \psi (x) + \frac{2}{R^2}\lambda \psi(x) \notag \\
		&&+ E_1\circ F_{p,R}(x) \label{PullBack}
\end{eqnarray}
for $p, R, \psi (x)$, where $\Delta_{S_1(0)} $ is the Laplacian on $S_1(0)$ with respect to the pull back metric and $\psi= \phi \circ F_{p,R} (x) = \phi( Rx + p)$ is the pull back of $\phi$. Define the operator $L : C^{2,\alpha}  \big( S_1(0) \big)  \rightarrow C^{0, \alpha} \big(S_1(0) \big)$ by
$$
		L \equiv - \Delta_0 - 2,
$$
where $\Delta_0 $ is standard spherical Laplacian on $S_1(0)$ in $\mathbb{R}^3$. Because the metric $g$ is asymptotically flat, the difference between $\Delta_0$ and $\Delta_{ S_1(0)}$ is small and can be treated as the error term. Therefore, the identity (\ref{PullBack}) is equal to
$$
		0 = \frac{ 6 m x \cdot p  }{  R^3} + \frac{ 9 m^2 }{ R^3 } + G\circ F_{p,R}(x) - \frac{1}{R^2} \lambda L\psi(x) + \widetilde{E}_1\circ F_{p,R}(x), 
$$
where $\widetilde{E}_1 $ has the same bound as $E_1$. We let $\lambda = R^{-a}$ for some fixed $a \in (0,1)$ and multiply $R^{ 2 + a}$ on both sides of the above equation, then
\begin{eqnarray}
		 L\psi(x)  =  \frac{ 6 m x \cdot p  }{  R^{ 1-a } } + \frac{ 9 m^2 }{ R^{ 1 - a } } + R^{2 + a } G\circ F_{p,R}(x) + R^{ 2 + a }\widetilde{E}_1\circ F_{p,R}(x) \label{PullBackL}
\end{eqnarray}
 Furthermore, since $D_x \psi = (D_y \phi )R$, 
\begin{eqnarray*}
		 | \widetilde{E}_1\circ F_{p,R}(x) | &\le& \frac{C}{R^4}\big(1 + |p| \big)+ \frac{C}{R^3} \bigg( \frac{ | \psi | }{ R^a}+ \frac{  | \psi |^2}{ R^{ 2a} } + \frac{| D^2 \psi ||\psi| }{ R^{2a}}\bigg).
\end{eqnarray*}

In order to find $p, R$ and $\psi$  to solve (\ref{PullBackL}), first we perturb $ p= p (R, \psi)$ so that the right hand side of (\ref{PullBackL}) is inside ${\rm Range}L$ for any $R$ and $\psi$. We will also show that $p = C_{CS} + e$ where $e$ is the error term containing lower order terms in $R$ and $\| \psi \|$. Second, using an iteration process and the Schauder estimate, we can find a solution $\psi$ for $R$ large. 
  
\noindent \textbf{1. Perturb the center $p$.} $L$ has a kernel equal to ${\rm span} \{ x^1, x^2, x^3\}$ because translation preserves the mean curvature. Since $L$ is self-adjoint, $C^{0, \alpha} ( S_1 (0) )$ has the $L^2$-orthogonal decomposition $C^{0, \alpha} ( S_1 (0) )= {\rm Range} {L} \oplus {\rm span} \{ x^1, x^2, x^3 \}$. We would like to find $p$ so that the right hand side of (\ref{PullBackL}) is orthogonal to ${\rm span} \{ x^1, x^2, x^3 \}$. That is, we want to find $p$ so that for  $\alpha = 1,2,3$,
\begin{eqnarray}
		&&\int_{S_1(0)} x^{\alpha} \Big( \frac{ 6 m x \cdot p  }{  R^{ 1 - a } } + \frac{ 9 m^2 }{ R^{ 1 - a } } \Big) \,d\sigma_0 \nonumber\\
		&&+ \int_{S_1(0)}x^{\alpha} \Big( R^{2 + a }G\circ F_{p,R}(x) + R^{ 2 + a }\widetilde{E}_1\circ F_{p,R}(x) \Big) \,d\sigma_0 = 0. \label{Perturbation}
\end{eqnarray}
We calculate each term above separately. A direct calculation gives the first integral 
$$
		\int_{S_1(0)} x^{\alpha} \Big( \frac{ 6 m x \cdot p  }{  R^{1-a}} + \frac{ 9 m^2 }{ R^{1-a} } \Big) \,d \sigma_0=  \frac{ 8m\pi p^{\alpha}}{R^{ 1 - a }}.
$$
From the area formula and Lemma \ref{Tech}, we have 
\begin{eqnarray*}
		&&\int_{S_1(0)} x^{\alpha} R^{2 + a }G\circ F_{p,R}(x)  \,d \sigma_0 = \int_{S_R (p) }\frac{ y ^{\alpha} - p^{\alpha} }{ R } R^{2 + a} G(y) R^{-2} \,d\sigma_0 \\
		&=& \frac{1 }{ R^{ 1-a} }  \int_{S_R (p) } ( y ^{\alpha} - p^{\alpha }) G(y) \,d\sigma_0 = \frac{ - 8m\pi C_{CS}^{\alpha}}{ R^{ 1 -a }} + O(R^{-2 + a}).
\end{eqnarray*}
Moreover, the error term can be bounded by
\begin{align*}
		\bigg| \int_{S_1(0)} x^{\alpha}	 R^{ 2 + a }\widetilde{E}_1\circ F_{p,R}(x) \,d\sigma_0  \bigg| &\le  \frac{c}{R^{2-a}}\big(1 + |p| \big)\\
		&+\frac{ c} {R} \big(| \psi | +  | D \psi | | \psi |+ | D \psi |^2+  | \psi| |D\psi| |D^2 \psi| \big).
\end{align*}
Since the ADM mass $ m > 0$, we can choose $p$ 
\begin{eqnarray}
		p (R, \psi ) = C_{CS} + e(R, \psi )  \label{Ctr}
\end{eqnarray}
so that the identity (\ref{Perturbation}) holds, where 
$$
		| e(R, \psi ) | \le  \frac{c}{R}\big(1 + |p| \big)+\frac{ c} {R^a} \big(| \psi | +  | D \psi | | \psi |+ | D \psi |^2+  | \psi| |D\psi| |D^2 \psi| \big).
$$

\noindent \textbf{2. Find the solution $\psi$ by iteration.} We consider the isomorphism $ L: ({\rm Ker}L)^{\perp} \rightarrow {\rm Range} (L)$ for $({\rm Ker}L)^{\perp} \subset C^{2,\alpha}(S_1(0)) $ and $ {\rm Range}(L) \subset C^{0,\alpha} (S_1(0))$. If we denote the right hand side of  (\ref{PullBackL})  by $f(p, R, \psi) $, i.e. 
$$
		f(p, R, \psi) =  - \frac{ 6 m x \cdot p  }{  R^{ 1-a } } - \frac{ 9 m^2 }{ R^{ 1 - a } } - R^{2 + a } G\circ F_{p,R}(x) - R^{ 2 + a }\widetilde{E}_1\circ F_{p,R}(x), 
$$
we know $ f(p(R, \psi), R, \psi) \in {\rm Range}(L)$ for any $R, \psi$ with $\| \psi \|_{C^{2, \alpha } } \le 1$. Therefore, any $\psi_0$ with $\| \psi_0 \|_{C^{2,\alpha}} \le 1$, there is $\psi_1 \in ( {\rm Ker} L)^{\perp}$ such that
$$
		L \psi_1 = f \big(p(R, \psi_0), R, \psi_0\big).
$$ 
Moreover, we use the Schauder estimate and the fact that $\psi_1\in ({\rm Ker}L)^{\perp}$,
\begin{eqnarray*}
		\| \psi_1 \|_{C^{2,\alpha}} &\le& c \left\| f \big(p(R, \psi_0), R, \psi_0\big) \right\|_{ C^{0,\alpha}} \\
		&\le & \frac{ c }{ R^{1 - a}} ( | C_{CS} | + 1 ) + \frac{ c} {R^a} (\| \psi_0 \|_{ C^{2, \alpha}} + \| \psi_0 \|^2_{ C^{2, \alpha}}  ) \\
		&\le & \frac{ c }{ R^{1 - a}} ( | C_{CS} | + 1 ) + \frac{2 c} {R^a}.
\end{eqnarray*}
For any $R$ large enough (independent of $\psi_0$), we have $ \| \psi_1 \|_{C^{2,\alpha}} \le 1$. We continue the iteration process and get a sequence of functions $\{ \psi_k \}_{k = 0}^{ \infty }$ satisfying
$$
		L \psi_{ k+ 1} = f\big( p(R, \psi_k ), R, \psi_k \big) \quad \mbox{ and }\quad \| \psi_{k+1} \|_{ C^{2, \alpha } } \le 1.
$$
The Arzela--Ascoli theorem says that there exists $\psi_{\infty} \in C^{2, \mu} (S_1(0))$ such that a subsequence of $\{ \psi_k \}$ converges to $\psi_{\infty}$ in $C^{2, \mu}$ for $ 0 < \mu < \alpha$. Moreover  $\psi_{\infty}$ is a solution to (\ref{PullBackL}),
$$	
		L \psi_{\infty} =  f \big(p(R, \psi_{\infty}), R, \psi_{\infty}\big).
$$
Let $\phi_{\infty} (y)= \psi_{\infty}\circ F^{-1}_{p,R}(y)$, then the surface $M_R = \{ z : z = y +  R^{-a} \phi_{\infty }  \nu, y \in S_{R} (p)\}$ has constant mean curvature equal to $ (2/R) - (4m/R^2)$.

To complete the proof of Theorem {\ref{HuiskenYau}}, we need to compute 
\begin{align} {\label{CHY}}
		\lim_{R \rightarrow \infty } \frac{ \int_{M_R} z^{\alpha} d \sigma_0  }{ \int_{ M_R} d\sigma_0}.
\end{align}
By the uniqueness of the constant mean curvature foliation, $\{M_R\}$ are equal to those constructed in \cite{HY}, and therefore (\ref{CHY}) converges to the Huisken--Yau center of mass $C^{\alpha}_{ HY}$. We now prove that (\ref{CHY}) also converges to the Corvino--Schoen center of mass $C^{\alpha}_{CS}$. Let $F$ be the diffeomorphism defined by $ F(y) = y + R^{-a} \phi_{\infty} \nu$, then 
\begin{eqnarray*}
				\frac{ \int_{M_R} z^{\alpha} d \sigma_0  }{ \int_{ M_R} d\sigma_0} = \frac{ \int_{S_R(p)} \big( y^{\alpha} + R^{-a}\phi_{\infty}  \nu^{\alpha}  ) JFd \sigma_0 }{ \int_{ S_R(p)} JF d\sigma_0},
\end{eqnarray*}
where $JF$ is the Jacobian from the area formula, $J F = 1 + O(R^{-1-a})$. Now we can use the fact that the area of Euclidean sphere is $O(R^2)$ and the estimate for the center $p$ in (\ref{Ctr}) to conclude
\begin{align*}
		&\frac{ \int_{M_R} z^{\alpha} d \sigma_0  }{ \int_{ M_R} d\sigma_0} \\
		&= p^{\alpha} +  \frac{ \int_{S_R(p)} \big( y^{\alpha} - p^{\alpha} ) (1 + O( R^{ -1-a} )) d\sigma_0 + \int_{S_R(p)} R^{-a} \phi_{\infty}  \nu^{\alpha}  JFd \sigma_0 }{ \int_{ S_R(p)} JF d\sigma_0} \\
		&= C^{\alpha}_{CS} + e(R,\phi_{\infty}) + O(R^{ -a } ).
\end{align*}
Therefore, after taking limits, the Huisken--Yau center of mass $C^{\alpha}_{ HY}$ is equal to the Corvino--Schoen center of mass $C^{\alpha}_{CS}$ and, therefore, is equal to $C^{\alpha}_I$.
\end{proof}

\section*{Acknowledgments}
I would like to thank my thesis advisor Prof. Richard Schoen for suggesting this problem and for the remarkable ideas and comments he provided. I also would like to thank Simon Brendle and Damin Wu for very useful discussions and Justin Corvino for his interest in this work.

\end{document}